\documentclass[a4paper, 11pt]{article}

\usepackage[utf8]{inputenc}
\usepackage[dvipsnames]{xcolor}

\usepackage{xfrac}
\usepackage[T1]{fontenc}
\usepackage{graphicx}

\usepackage{amsmath}
\usepackage{amsthm}
\usepackage{amsbsy}
\usepackage{amssymb}
\usepackage{mathtools}
\usepackage{upgreek}

\usepackage{algorithm}
\usepackage[noend]{algpseudocode}

\makeatletter
\def\BState{\State\hskip-\ALG@thistlm}
\makeatother

\usepackage{graphpap}
\usepackage[pdftex]{hyperref}
\usepackage[top=4cm, bottom=4cm, left=2.5cm, right=2.5cm]{geometry}
\usepackage{version} 

\usepackage[toc,page]{appendix}

\newcommand{\R}{\mathbb{R}}
\newcommand{\C}{\mathbb{C}}
\newcommand{\Class}[1]{\mathcal{C}^{#1}}
\newcommand{\Z}{\mathbb{Z}}
\newcommand{\Tau}{T}

\newcommand{\dd}{\partial}
\DeclareMathOperator{\distance}{d}
\newcommand\dist[3]{\distance\left(#1,#2\right)_{#3}}

\newcommand{\larrow}[1]{\overset{\text{\tiny$\leftarrow$}}{#1}}
\newcommand{\darrow}[1]{\overset{\text{\tiny$\leftrightarrow$}}{#1}}

\newcommand{\ehat}[1]{\expandafter\hat#1}
\newcommand{\coorientation}[1]{\mathcal{E}ul\mathcal{C}o(#1)}
\newcommand{\cor}{c}

\theoremstyle{plain}
\newtheorem{theorem}{Theorem}

\newtheorem{remark}[theorem]{Remark}
\newtheorem{example}[theorem]{Example}
\newtheorem{examples}[theorem]{Examples}
\newtheorem{remarks}[theorem]{Remarks}

\newtheorem{lemma}[theorem]{Lemma}
\newtheorem{proposition}[theorem]{Proposition}

\newtheorem{introtheorem}{Theorem}
\newtheorem{introcorollary}[introtheorem]{Corollary}

\theoremstyle{definition}
\newtheorem{definition}[theorem]{Definition}

\newtheorem*{introresult}{Main results}
\newtheorem*{mainconventions}{Main conventions}

\DeclareMathOperator{\Id}{Id}

\DeclareMathOperator{\inte}{int}

\DeclareMathOperator{\supp}{supp}
\DeclareMathOperator{\e}{e}

\title{First-return maps of Birkhoff sections of the geodesic flow}
\author{Th\'eo MARTY}

\begin{document}

\maketitle

\begin{abstract}
This paper compares different pseudo-Anosov maps coming from different Birkhoff sections of a given flow. More precisely, given a hyperbolic surface and a collection of periodic geodesics on it, we study those Birkhoff sections for the geodesic flow on the unit bundle to the surface bounded by the collection. We show that there is a canonical identification of all those Birkhoff sections, and that the first-return maps induced by the flow can all be expressed as a composition of negative Dehn twists along a family of explicit curves~: only the order depends on the choice of a particular Birkhoff section. 
\end{abstract}

\tableofcontents

\section*{Introduction}
\addcontentsline{toc}{section}{Introduction}

The unit sphere of the Thurston norm of a compact~$3$-manifold~$M$ is a polyhedron in~$H_2(M,\R)$. To a fibration~$M\to S^1$ by compact surfaces corresponds a rational point in the sphere given by homology ray containing the fibers. By a theorem of Thurston, a flow in~$M$ corresponds to a  so-called  fibered face in the unit sphere, given by all fibrations whose fibers are global sections for the flow. The flow also induces first-return maps on these global sections. The goal of this paper is to understand how these first-return maps are all connected, in a specific case.

When the flow is of pseudo-Anosov type, the first-return on a section~$S$ is of pseudo-Anosov type. In particular it has a dilatation factor~$K>1$. Fried studied~\cite{Fried1982} the function~$\chi(S)\ln(K)$, which is convex and tends to infinity on the boundary of the fibered face. McMullen defined~\cite{McMullen00} the Teichmüller polynomial in~$\Z[H^1(M,\Z)]$, whose specialization at an integral point has~$K$ as greater root. This paper goes in the same direction by giving, for one explicit family of fibered faces, a computation and a comparison of the first-return maps, as products of Dehn twists.  

We are interested in the geodesic flow of a hyperbolic surface, which is an Anosov flow. Once one removes finitely periodic orbits, we obtain a pseudo-Anosov flow on a~$3$-manifold with toric boundary. The global sections for such flows come from Birkhoff sections of the original flow. Under a certain symmetry assumption, the fibered faces for these are rather well-understood. 

\begin{introresult}
Given a hyperbolic surface~$S$ and a symmetric collection~$\darrow\Gamma$ of periodic orbits of the geodesic flow on the unit tangent bundle~$T^1S$, there is a common combinatorial model~$\Sigma_\Gamma$ for all Birkhoff sections with boundary~$-\darrow{\Gamma}$, and a finite collection~$\gamma_1, \cdots, \gamma_n\subset \Sigma_\Gamma$ of simple closed curves on~$\Sigma_\Gamma$ such that the first-return map along the geodesic flow is of the form~$\tau_{\gamma_{\sigma(1)}}^{-1}\circ\dots\circ\tau_{\gamma_{\sigma(n)}}^{-1}$ for some permutation~$\sigma$ of~$\{1, \cdots, n\}$. The permutation~$\sigma$ depends explicitly on the point in the fibered face. Here~$\tau_\gamma^{-1}$ denotes the negative Dehn twist along~$\gamma$. 
\end{introresult}

These results will be restated below in course of the introduction as Theorems~\ref{introtheoremA},~\ref{introtheoremB},~\ref{introtheoremC} and Corollary \ref{introcorollaryD}. The product of negative Dehn twists gives a way to explicitly compare first-return maps for different integer points of the same fibered face, by only changing the order of the Dehn twists.

These results are reminiscent of A'Campo's divide construction~\cite{A'Campo98}, and of Ishikawa's generalization \cite{Ishikawa2004}. They decompose a monodromy as an explicit product of three Dehn multi-twists.  A'Campo's result was also recently generalized by Dehornoy and Liechti~\cite{dehornoy2019divide} who expressed the monodromy for divide links in the unit tangent bundle of arbitrary surfaces as products of two antitwists.  Our results deal more generally with all integral points in the fibered face, instead of just the center.

\paragraph{Birkhoff sections.} Let~$S$ be a hyperbolic closed surface with a fixed hyperbolic metric on~$S$, and let~$\phi$ be the geodesic flow on~$T^1S$. We will study some properties of the flow~$\phi$. For the rest of the paper,~$S$ and~$\phi$ will denote this hyperbolic surface and its geodesic flow. We are interested in finding Birkhoff sections, which are compact embedded surfaces~$\Sigma\subset T^1S$ such that:
\begin{itemize}
\item the interior of~$\Sigma$ is transverse to~$\phi$,
\item there exists~$t>0$ such that~$\phi_{[0,t]}(\Sigma)=T^1S$ (every orbit reaches~$\Sigma$ after a bounded time),
\item~$\partial \Sigma$ is a finite union of closed orbits of~$\phi$.
\end{itemize}

We call~$\Sigma$ a \textbf{transverse surface} if only the first and third points are satisfied. For a Birkhoff section~$\Sigma$, we denote by~$r_\Sigma:\inte(\Sigma)\to\inte(\Sigma)$ the induced first-return map.

\paragraph{Birkhoff sections with symmetric boundary.} Let~$\Gamma\subset S$ be a closed geodesic multi-curve composed by~$n$ curves. For the rest of the article, we suppose that~$\Gamma$ is \textbf{filling}, that is,~$S\setminus\Gamma$ is a union of discs. We also suppose that~$\Gamma$ is in generic position, that means with only degree~$4$ intersections, as in Figure \ref{example1}. This multi-curve lifts in~$T^1S$ into a multi-curve~$\darrow\Gamma$ of~$2n$ closed orbits of~$\phi$, where each curve of~$\Gamma$ is lifted with both orientations. We fix on~$\darrow\Gamma$ the orientation given by the geodesic flow. We will study the Birkhoff sections~$\Sigma$ with~$\partial\Sigma=\darrow{\Gamma}$, and whose multiplicities along their boundaries are~$(-1,\cdots,-1)$. It means that the usual orientation on~$T^1\Sigma$ and the coorientation of~$\Sigma$ by the flow induce an orientation on~$\Sigma$, which induces on~$\partial\Sigma$ the orientation opposite to the flow. Such a Birkhoff section is said to be bounded by~$-\darrow\Gamma$, and to have \textbf{symmetric boundary}.

\begin{figure}[ht]
\centering
\includegraphics[scale=0.8]{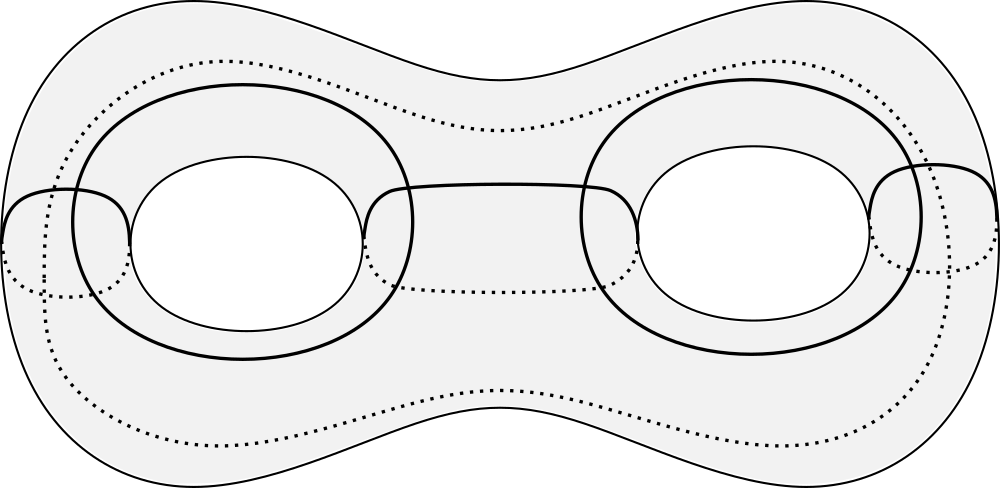}
\caption{Filling geodesic multi-curve on a hyperbolic surface.}
\label{example1}
\end{figure}

In Section \ref{section:construction}, we construct explicitly a surface~$\Sigma$ transverse to the flow, that is represented in Figure \ref{fig:introlift}. It relies on the choice of a Eulerian coorientation~$\eta$ of~$\Gamma$. Elementary properties and diffeomorphisms will be expressed using the combinatorics of~$\eta$, for example~$\Sigma_\eta$ is a Birkhoff section if and only if there is no oriented cycle in the dual graph~$(\Gamma^\star,\eta)$.

According to the classification of Birkhoff sections with symmetric boundaries of~\cite{CossariniDehornoy}, every Birkhoff section bounded by~$-\darrow\Gamma$ is isotopic to one such~$\Sigma_\eta$.

\begin{figure}[ht]
\centering
\includegraphics[scale=0.25]{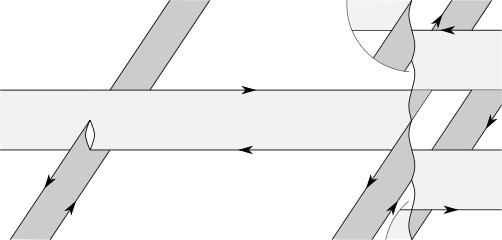}
\caption{Local picture of the surface~$\Sigma_\eta\subset T^1S$ locally identified with~$\R^2\times S^1$.} \label{fig:introlift}
\end{figure}

The surface~$\Sigma_\eta$ stays mainly in some specific fibers of~$\pi:T^1S\to S$, and~$\pi_{|\Sigma_\eta}$ is not an immersion. In order to make~$\Sigma_\eta$ easier to use, we deform it into an immersed surface.

\begin{introtheorem}\label{introtheoremA}
Given a geodesic multi-curve~$\Gamma$ and an Eulerian coorientation~$\eta$ of~$\Gamma$, there exists a small isotopy~$(f_t)_t$ of the associated surface~$\Sigma_\eta$ such that~$f_0=\iota_{|\Sigma_\eta \xhookrightarrow{} T^1S}$ and~$\pi\circ f_1:\Sigma_\eta\to S$ is an immersion (see Figure \ref{fig:exampleisotopy}).
\end{introtheorem}

We will study several representation of~$\Sigma_\eta$ in Section \ref{section:construction}. We are precisely interested by the immersion of Theorem \ref{introtheoremA}, and by the ribbon graph representation it induces.

\begin{figure}[ht]
\centering
\includegraphics[scale=0.13]{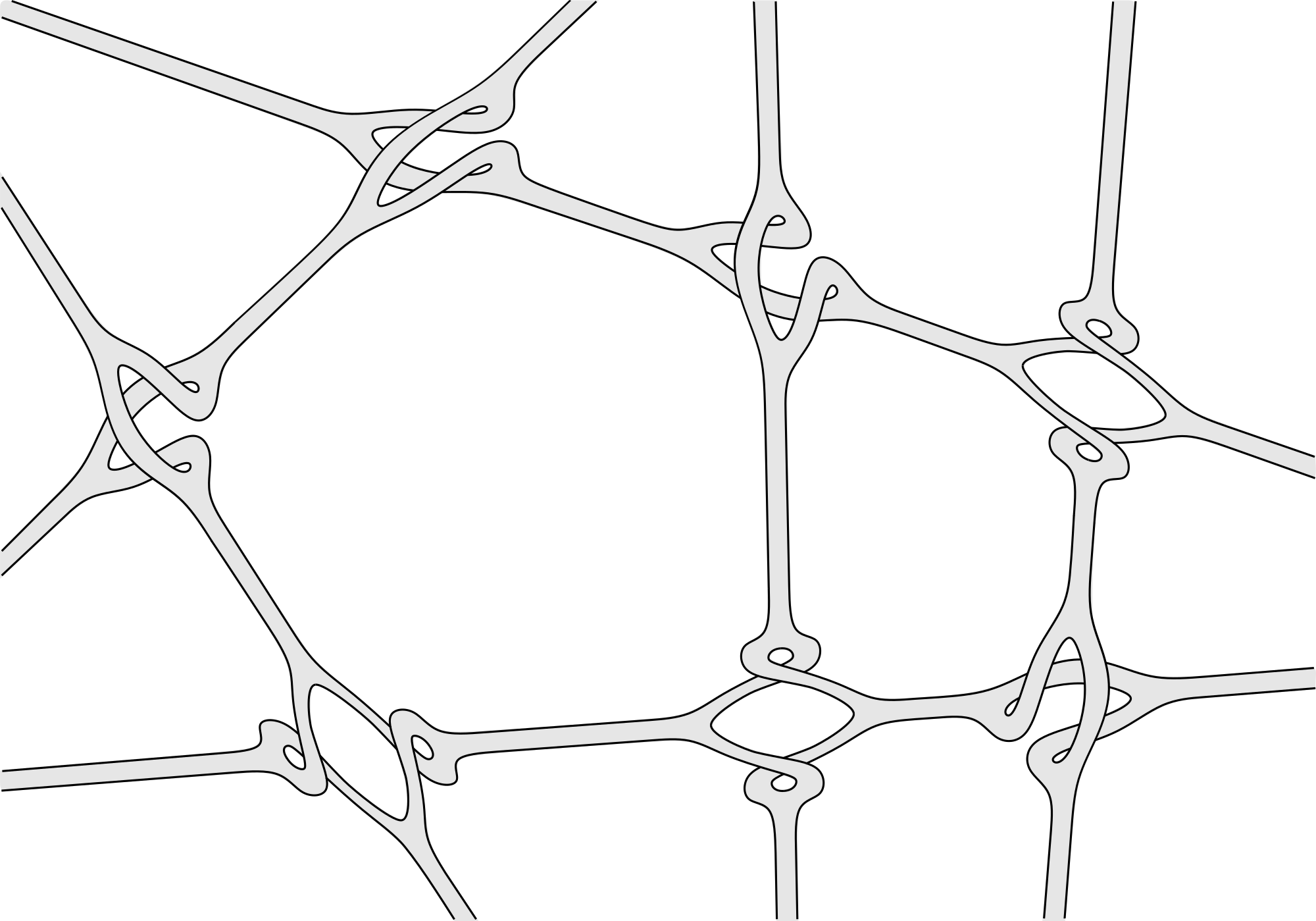}
\caption{A small perturbation of~$\Sigma_\eta$ into a horizontal surface.}
\label{fig:exampleisotopy}
\end{figure}

\paragraph{Partial return maps.} The main idea for computing the first-return map~$\Sigma_\eta\to\Sigma_\eta$ is to define intermediate disjoint and homologous Birkhoff sections~$(\Sigma_i)_i$, so that the first-return map~$r_{\Sigma_\eta}$ is the composition of partial return maps ~$$\Sigma_\eta=\Sigma_0\to\Sigma_1\to\hdots\to\Sigma_{n-1}\to\Sigma_n=\Sigma_\eta$$

We define the surfaces~$\Sigma_i$ by induction using elementary transformations, so that~$r_i:\Sigma_{i-1}\to\Sigma_{i}$ is quite simple to compute. These elementary transformations have a combinatorial and a geometric version. The combinatorial version consists in taking an Eulerian coorientation~$\eta$ and modifying it around one specific face, thus obtaining a new coorientation~$\eta'$. The surfaces~$\Sigma_\eta$ and~$\Sigma_{\eta'}$ are isotopic and easy to compare. If we do this transformation around each face in the right order, we describe a cyclic family of Birkhoff sections~$(\Sigma_i)_{0\leq i\leq n}$, pairwise easily comparable.

The geometric version of this transformation consists in taking~$\eta$ and~$\eta'$ that differ around a face~$f$, and following the flow only in~$T^1_fS$. It describes a map~$\Sigma_\eta\to\Sigma_{\eta'}$ that we call \textbf{partial return map}. The partial return maps together with the family of Birkhoff sections~$(\Sigma_i)_i$ allow us to reconstruct the first-return map.

\begin{introtheorem}\label{introtheoremB}
Let~$\Gamma\subset S$ be a filling geodesic multi-curve of a hyperbolic surface~$S$,~$\eta$ an acyclic Eulerian coorientation of~$\Gamma$ and~$f_1,\hdots,f_n$ be the faces of~$S\setminus\Gamma$, ordered by~$\eta$. Then the first-return map along the geodesic flow on the Birkhoff section~$\Sigma_\eta$ is given by~$r_{\Sigma_\eta}=r_n\circ\hdots\circ r_1$, where~$r_i$ is the partial return map along the face~$f_i$.
\end{introtheorem}

In this theorem, we order the faces so that if~$\eta$ goes from the face~$f_i$ to~$f_j$ around an edge, then~$j<i$. We study this elementary transformation in Section \ref{section:combiflip}, together with the combinatorial tools needed to express precisely Theorem \ref{introtheoremC}.

\paragraph{Explicit first-return map.}

To compute the first-return map, we need to compute explicitly the partial return maps. Fix~$r_i:\Sigma_{i-1}\to\Sigma_i$ a partial return map. We would like to compose~$r_i$ with a nice \textbf{correction} function~$\cor_i$ so that the composition~$\Sigma_{i-1}\xrightarrow{r_i}\Sigma_{i}\xrightarrow{\cor_i}\Sigma_{i-1}$ is a Dehn twist. We will use the ribbon representation of~$\Sigma_{i-1}$ and~$\Sigma_{i}$ to compare them, especially around the vertices at which they differ. After defining~$\cor_i$, the composition~$\cor_i\circ r_i$ is isotopic to a negative Dehn twist along the curve~$\gamma_f$, as shown in Figure~\ref{fig:curves}.

\begin{figure}[ht]
\centering
\begin{center}
\begin{picture}(100,32)(0,0)
\put(0,0){\includegraphics[width=100mm]{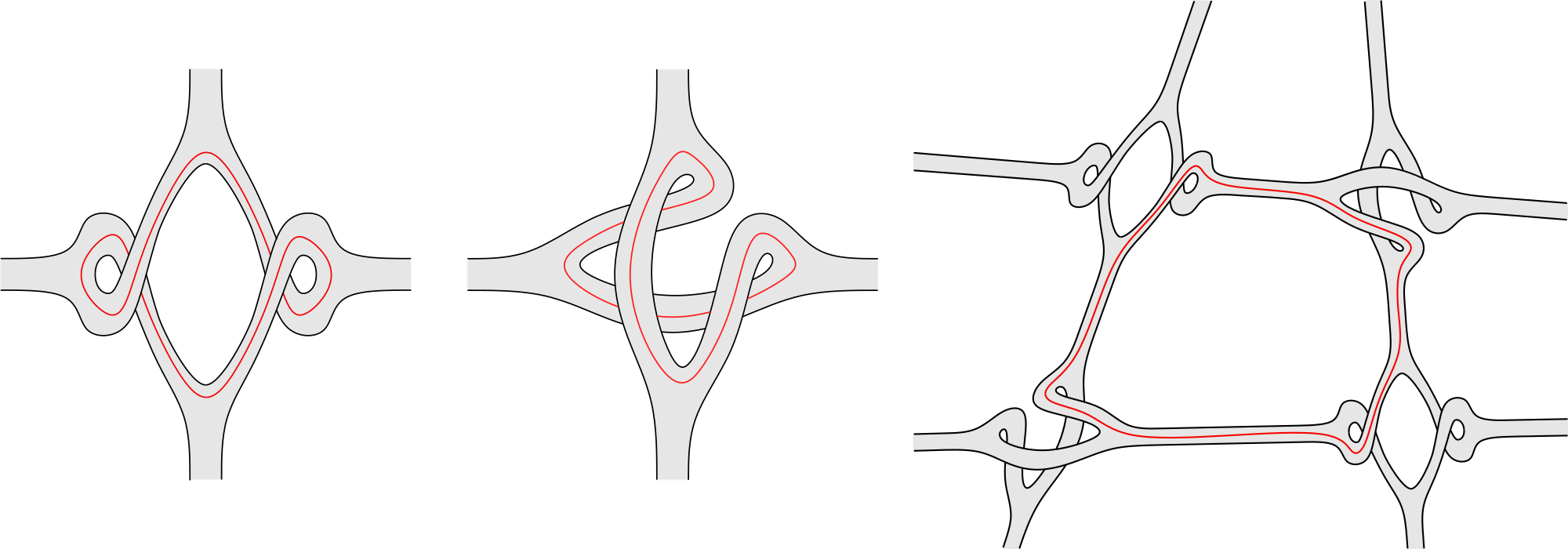}}
\put(78.5,14.5){$f$}
\color{red}
\put(11.5,17.5){$\gamma_v$}
\put(42,17.5){$\gamma_v$}
\put(91,14){$\gamma_f$}
\color{black}
\end{picture}
\end{center}
\caption{Curves~$\gamma_v$ for a vertex~$v$ and~$\gamma_f$ for a face~$f$.}
\label{fig:curves}
\end{figure}
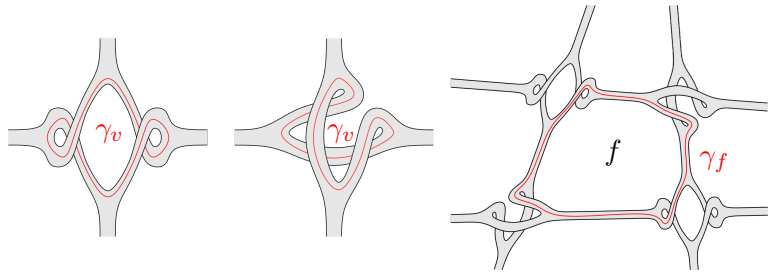

Figure \ref{fig:curves} only shows~$\gamma_f$ for a sink face. A complete description of~$\gamma_f$ in the general case is done in Section \ref{section:first_return_map}. This computation, together with Theorem \ref{introtheoremB}, allows to compute the first-return map as a product of negative Dehn twists.

\begin{introtheorem}\label{introtheoremC}
Let~$\eta$ be an acyclic Eulerian coorientation and~$\Sigma_\eta$ its corresponding Birkhoff section. Then the first-return map~$r:\Sigma_\eta\to\Sigma_\eta$ is the product of explicit negative Dehn twists along the explicit curves~$\gamma_v$ and~$\gamma_f$ for all~$v\in\Gamma_0$ and~$f\in\Gamma_0^\star$. The order of the Dehn twists is given by~$\eta$.
\end{introtheorem}

A precise statement and a proof of this theorem will be given in Section \ref{section:first_return_map}.

\begin{introcorollary}\label{introcorollaryD}
Let~$S$ be a hyperbolic surface,~$\Gamma$ a finit collection of closed geodesics on~$S$, and consider the geodesic flow on~$T^1S$. There exists a common combinatorial model~$\Sigma_\Gamma$ for all Birkhoff sections with boundary~$-\darrow{\Gamma}$, and an explicit family of simple closed curves~$\gamma_1, \cdots, \gamma_n$ in~$\Sigma_\Gamma$ such that the first-return maps for these Birkhoff sections are of the form~$\tau_{\gamma_{\sigma(1)}}^{-1}\circ\dots\circ\tau_{\gamma_{\sigma(n)}}^{-1}$ for some permutation~$\sigma$ of~$\{1, \cdots, n\}$. 
\end{introcorollary}

In Theorem \ref{introtheoremC}, the Birkhoff sections and the curves supporting the Dehn twists are explicit, and only depend on the choice of one coorientation. Also the ordering of the Dehn twists is almost canonical. In Corollary \ref{introcorollaryD}, there are only one abstract Birkhoff surface and collection of curves, that are also explicit. But the ordering of the negative Dehn twists and the first return map is less explicit, and need more work to be constructed by hand.

\paragraph{Example on a flat torus.} On a flat torus, the classification of Birkhoff sections is different, and can be found in \cite{Dehornoy2015}. However the surfaces~$\Sigma_\eta$ can be defined similarly and they are Birkhoff sections. Also Theorem~$A$,~$B$ and~$C$ are still true for these surfaces. However it is simpler to illustrate them on the torus.

In Figure \ref{fig:fullexample}, we briefly illustrate the theorems on the flat torus, given by a square whose opposite sides are identified. Let~$\Gamma$ and~$\eta$ be the multi-curve and the coorientation given on the picture. Theorem \ref{introtheoremA} gives an immersion of the Birkhoff section~$\Sigma_\eta$ into the torus, which is represented on the right. Four examples of the curves~$\gamma_f$ (in red) and~$\gamma_v$ (in blue and green) are also represented.

We order, from~$1$ to~$12$, the vertices of~$\Gamma$ and the faces it delimitates. For this, we complete the natural order given by the coorientation~$\eta$ of the faces, using additional rules explained in Section \ref{section:combiflip}. Theorem \ref{introtheoremC} then states that the first-return map~$r$ on~$\Sigma_\eta$ is a product of negative Dehn twists, with the order previously chosen. So that if~$T\gamma$ denotes the negative Dehn twist along~$\gamma$, then:
\[r_{\Sigma_\eta} = T\gamma_{f_{12}}\circ T\gamma_{f_{11}}\circ T\gamma_{v_{10}}\circ T\gamma_{v_9}\circ T\gamma_{v_8}\circ T\gamma_{v_7}\circ T\gamma_{f_6}\circ T\gamma_{f_5}\circ T\gamma_{v_4}\circ T\gamma_{v_3}\circ T\gamma_{f_2}\circ T\gamma_{f_1} \]

\begin{figure}[ht]
\centering
\begin{center}
\begin{picture}(150, 100)(0,0)
\put(0,0){\includegraphics[width = 150mm]{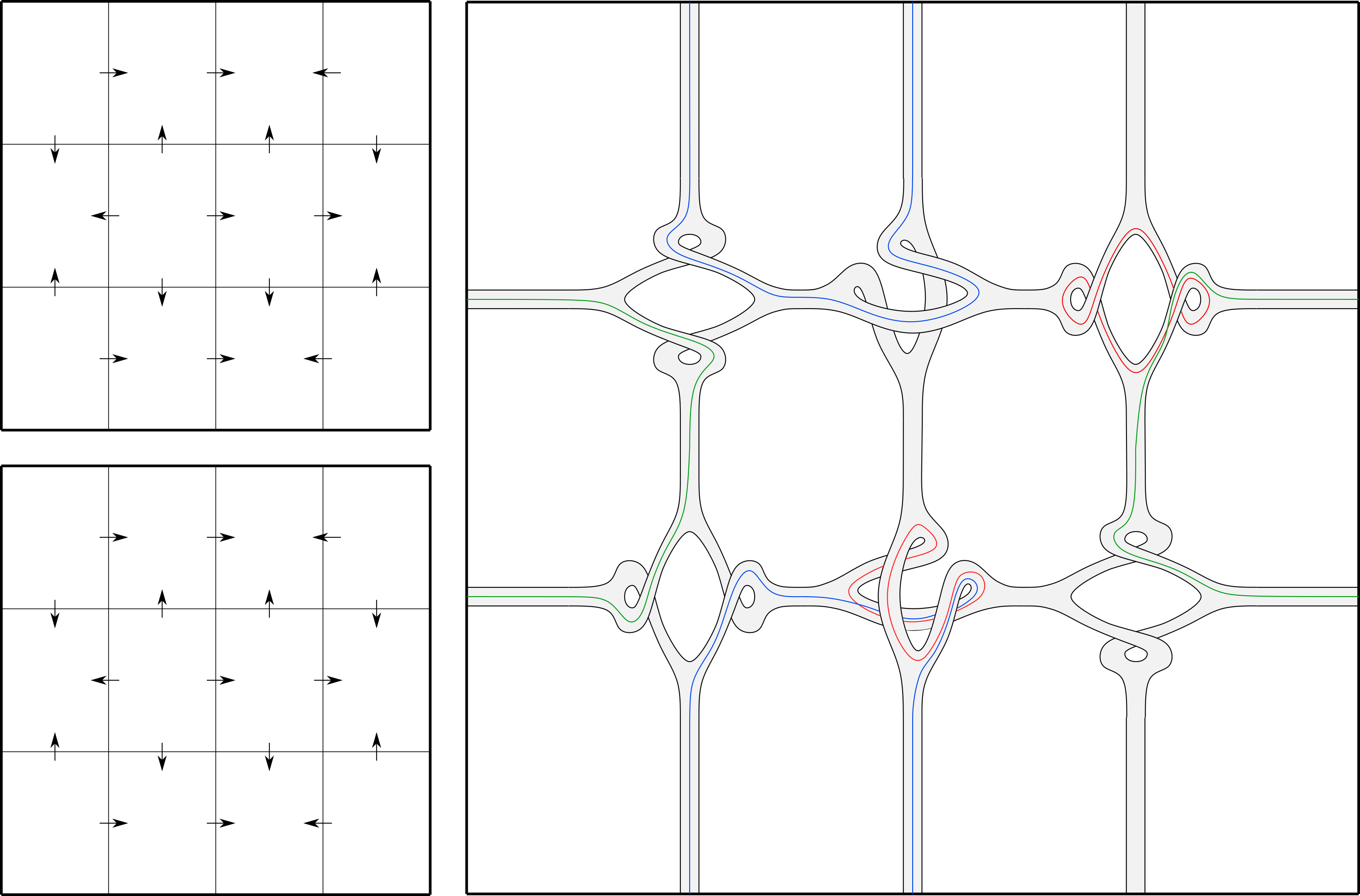}}

\put(4, 38.5){$12$}
\put(17, 38.5){$6$}
\put(29, 38.5){$2$}
\put(5.2, 22.5){$1$}
\put(16, 22.5){$11$}
\put(29, 22.5){$5$}
\put(11, 30.5){$9$}
\put(23, 30.5){$7$}
\put(34.5, 30.5){$3$}
\put(10, 14.5){$10$}
\put(23, 14.5){$8$}
\put(34.5, 14.5){$4$}

\color{gray!30!white}
\put(4, 7){$12$}
\put(17, 7){$6$}
\put(29, 7){$2$}
\put(40, 7){$12$}
\put(40, 38.5){$12$}
\put(40.7, 22.5){$1$}
\color{black}

\end{picture}
\end{center}

\caption{Example on a flat torus~$\sfrac{\R^2}{\Z^2}$.}
\label{fig:fullexample}
\end{figure}

I am grateful to P.Dehornoy for introducing me to the subject, and together with E.Lanneau for the continuous discussions and remarks. I thank Burak \"Ozba\u{g}ci for the interesting discussions and remarks.

\newpage
\setcounter{introtheorem}{0}

\section{Representations of the Birkhoff sections~$\Sigma_\eta$\label{section:construction}}

\begin{mainconventions}\label{convention}
	In this article, we will focus on the following assumptions, which allow us to study the first-return map on the Birkhoff section~$\Sigma_\eta$ (constructed in Section~\ref{subsection:construction}). We fix a hyperbolic surface~$S$,~$\phi$ the geodesic flow on~$T^1S$,~$\Gamma\subset S$ a filling geodesic multi-curve in generic position, and~$\eta$ an Eulerian coorientation (defined in Section \ref{subsection:construction}) such that~$\Sigma_\eta$ is a Birkhoff section of~$\phi$.  The choice of the hyperbolic metric on~$S$ has a very little influence on what we will discuss, only the combinatorics of~$\Gamma$ matters.  
\end{mainconventions}

Starting from~$\eta$, an object called a coorientation of~$\Gamma$, we construct a Birkhoff section~$\Sigma_\eta$ of the geodesic flow~$\phi$. We then find good representations of~$\Sigma_\eta$, including a ribbon graph representation. It will later help us to do explicit computations. Our two goals in this section are to prove Theorem \ref{introtheoremA} and to study some elementary properties of the ribbon representation.

\subsection{Construction of~$\Sigma_\eta$ \label{subsection:construction}}

In this subsection, we construct the surface~$\Sigma_\eta$. This construction and its first properties come from \cite{CossariniDehornoy}.

See~$\Gamma$ as a graph~$(\Gamma_0,\Gamma_1)$ in~$S$, were~$\Gamma_0$ is the set of double points of~$\Gamma$, and~$\Gamma_1$ the set of edges bounded by~$\Gamma_0$. We also denote by~$\Gamma_2$ the set of faces of~$S$ bounded by~$\Gamma$. We consider a coorientation~$\eta$ of~$\Gamma$, in the sense that~$\eta$ is the union of a transverse orientation for every edge in~$\Gamma_1$ (see Figure \ref{fig:fullexample} left). We are interested in \textbf{Eulerian coorientations}, that is, around every vertex there are as many edges locally oriented clockwisely and anticlockwisely. In particular, around a vertex, there are two ways to coorient~$\Gamma$ up to rotation, that we call the alternating and non-alternating vertices (see Figure \ref{fig:vertex}).

\begin{definition}
We denote by~$\coorientation{\Gamma}$ the set of all Eulerian coorientations of~$\Gamma$.
\end{definition}

\begin{figure}[ht]
\centering
\begin{center}
\begin{picture}(80,45)(0,0)
\put(0,5){\includegraphics[width = 80mm]{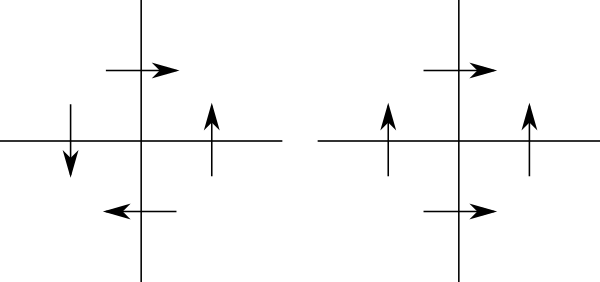}}
\put(10,0){alternating}
\put(47,0){non-alternating}
\end{picture}
\end{center}
\caption{Eulerian coorientation around a vertex.}
\label{fig:vertex}
\end{figure}

\begin{examples}
\begin{itemize}
\item We can coorient every geodesics of~$\Gamma$ and combine them in an Eulerian coorientation of~$\Gamma$, with only non-alternating vertices.
\item If~$[\Gamma]\equiv 0\in H^1(S,\sfrac{\Z}{2\Z})$, we can color the faces of~$\Gamma\subset S$ in black and white, and then define the coorientation that goes from white to black. It has only alternating vertices. This is a historical example which corresponds to a previous work by Birkhoff \cite{Birkhoff314}.  In \cite{A'Campo98} and \cite{Ishikawa2004}, N.A'Campo and M.Ishikawa computed the first-return map for this choice of coorientation, in similar contexts, and proved that it is a product of three Dehn multi-twists. The curves supporting these twists correspond to the white faces, the double points of~$\Gamma$, and to the black faces. These curves will also appear in our construction. 
\end{itemize}
\end{examples}

We now fix an Eulerian coorientation~$\eta$ and construct the surface~$\Sigma_\eta$. The first step is to define a vertical~$2$-complex~$\ehat{\Sigma_\eta}$ in~$T^1S$. For every edge~$e\in\Gamma_1$, let~$r_e=\{(x,v)\in T^1S| x\in\Gamma,\text{$v$ and~$\eta_e$ are in the same direction}\}$ be a vertical rectangle (see Figure \ref{fig:rectangle}). Then define the~$2$-complex~$\ehat{\Sigma_\eta} = \cup_{e\in\Gamma^1}r_e$. Apart from the fibers of the non-alternating vertices, it is a topological surface with boundary~$-\darrow\Gamma$.

Let~$v$ be an alternating vertex of~$\Gamma$. On the fiber~$T^1_vS$, the complex~$\ehat{\Sigma_\eta}$ admits a degree~$4$ edge, as a~$X$-shape times~$[0,1]$. We need to resolve this singularity. There are two ways to desingularise and smooth~$\ehat{\Sigma_\eta}$ into a surface around~$T^1_vS$, but only one is transverse to~$\phi$. We desingulerise~$\ehat\Sigma_\eta$ and define~$\Sigma_\eta$ the smoothing of~$\ehat{\Sigma_\eta}$ into a surface transverse to~$\phi$. A local lift of~$\Sigma_\eta$ to~$\R^2\times \R$ is represented in Figure \ref{fig:representations}.c. This surface is unique up to a small isotopy along the flow. To simplify forthcoming expressions, we denote by~$\Sigma_\eta$ the interior of the surface, but we still consider its boundary~$\partial\Sigma_\eta=-\darrow\Gamma$.

\begin{remark}
We will see that the diffeomorphism class of the surface~$\Sigma_\eta$ does not depend on the type of vertices induced by the coorientation~$\eta$. Thus it does not depend on the coorientation~$\eta$ itself. However its isotopy type inside~$T^1S$ depends on~$\eta$, as explained at the end of the section.
\end{remark}

Given~$X\subset T^1S$, denote by~$i_{|X}$ the inclusion of~$X$ into~$T^1S$. For a given property~$\mathcal{P}$, we say that there exist a \textbf{small} map~$h:\Sigma_\eta\to T^1S$ that satisfies~$\mathcal{P}(h)$ if:~$\forall\epsilon>0,\exists\mu>0$ such that if a smoothing~$s:\ehat{\Sigma_\eta}\to\Sigma_\eta$ satisfies~$\dist{i_{|\Sigma_\eta}\circ s}{i_{|\ehat{\Sigma_\eta}}}{\Class{1}}<\mu$, then there exists a diffeomorphism~$h$ that satisfies~$\mathcal{P}$ such that~$\dist{h}{i_{|\Sigma_\eta}}{\Class{1}}<\epsilon$. We extend this vocabulary for isotopies.

\paragraph{Classification of the Birkhoff sections with boundary~$-\darrow\Gamma$.}

Given a coorientation~$\eta$ and a generic closed curve~$\gamma$ in~$S$, we can count the algebraic intersection between~$(\Gamma,\eta)$ and a curve~$\gamma$, which we write~$\eta(\gamma)$.

\begin{lemma}\cite{CossariniDehornoy}
If~$\eta$ is Eulerian,~$\eta(\gamma)$ depends only on the homology class~$[\gamma]\in H_1(S,\Z)$. Thus the coorientation~$\eta$ induces a cohomology class~$[\eta]\in H^1(S,\Z)$.
\end{lemma}

It is known that Birkhoff sections are classified up to isotopy by their homology class (see \cite{Fried82section} or \cite{Schwartzman1957}). Given an Eulerian coorientation~$\eta$, its cohomology is used in \cite{CossariniDehornoy} to classify the Birkhoff surfaces with symmetric boundary~$-\darrow\Gamma$. According to Theorems C and D of \cite{CossariniDehornoy}, the set of relative homology class~$[\Sigma_\eta]$ realizes every relative homology class of transverse surface with boundary~$-\darrow{\Gamma}$. In particular every Birkhoff section bounded by~$-\darrow\Gamma$ is isotopic to one~$\Sigma_\eta$ for some~$\eta$. For two Eulerian coorientations so that~$\Sigma_\eta$ is a Birkhoff section,~$[\eta] = [\nu]$ in cohomology if and only if~$\Sigma_\eta$ and~$\Sigma_\nu$ are isotopic through the geodesic flow~$\phi$. Additionally the set of~$[\eta]$ is a convex polyhedra inside~$H^1(M,\Z)$, and~$\Sigma_\eta$ is a Birkhoff section if and only if~$[\eta]$ lies in the interior of this polyhedra.

The set of Birkhoff sections with fixed symmetric boundary is a polyhedra not completely understood, but we can describe explicitly every points it contains. Indeed given~$\sigma\in H^1(S,\Z)$, there is a procedure that constructs, if it exists,~$\eta\in\coorientation{\Gamma}$ such that~$\sigma=[\eta]$ (see Appendix \ref{appendix: construction}).

\paragraph{Skeleton of~$\Sigma_\eta$}

We use Figure \ref{fig:rectangle} to define a skeleton~$\ehat X$ of~$\ehat\Sigma_\eta$, that will be pushed into a skeleton~$X$ of~$\Sigma_\eta$. Take an edge~$e\in\Gamma_1$. It corresponds to a flat rectangle~$r_e$ in~$\ehat\Sigma_\eta$, that is isometric to~$e\times [-1,1]$. Denote by~$\{v_1,v_2\}=\partial e$ and~$a_1,a_2\in(-1,1)$ the angle between~$e$ and the intersection geodesic on~$v_1$ and~$v_2$. The rectangle~$r_e$ is attached to four other rectangles (counting with multiplicity) on the four segments given by~$v_i\times[-1,a_i]$ and~$v_i\times[a_i,1]$. So we put a vertex in the middle of each of these segments, and we connect them as in Figure \ref{fig:rectangle}. The union for every~$e\in\Gamma_1$ defines a skeleton~$\hat X$ of~$\ehat\Sigma_\eta$, that we push into~$\Sigma_\eta$ to define the skeleton~$X$.

\begin{figure}[ht]
\centering
\begin{center}
\begin{picture}(55,34)(0,0)
\put(0,0){\includegraphics[width=60mm]{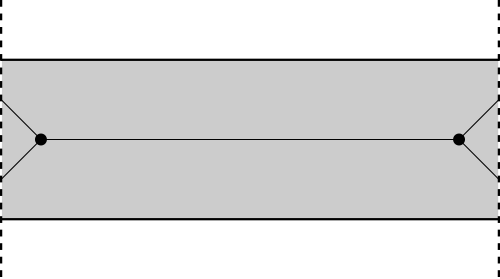}}
\put(29,28){$\overrightarrow{e}$}
\put(29,3){$\overleftarrow{e}$}
\put(0.8,28){$T^1_{v_1}S$}
\put(51.5,28){$T^1_{v_2}S$}
\put(29,18){$\ehat{X}$}
\end{picture}
\end{center}
\caption{Skeleton of~$\ehat{\Sigma_\eta}$ in the rectangle~$r_e$} \label{fig:rectangle}
\end{figure}
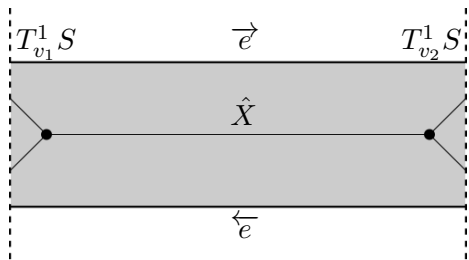

Notice that locally around every vertex~$v\in\Gamma_0$, the skeleton~$X$ is homeomorphic to a circle glued once to four edges leaving the circle, independently of the nature of~$v$. To be more precise, we describe~$\pi(X)$ the projection of~$X$ into~$S$. If the smoothing of~$\Sigma_\eta$ is well chosen,~$\pi(X)$ can be obtained from~$\Gamma$ by replacing each alternating vertex by a square, and each non-alternating vertices by a twisted square, as in Figure \ref{fig:representations}.b.

\subsection{Isotopies and immersion of~$\Sigma_\eta$.}
The definition of~$\Sigma_\eta$ makes it a bit hard to compute algebraic and geometric intersections between explicit curves. We will give two other descriptions, obtained by isotopy, of~$\Sigma_\eta$ that will help us.

\paragraph{Isotopy with an immersion. \label{smallIsotopy}}

The isotopy of~$T^1S$ that interests us is the parallel transport that pushes~$(x,u)$ in the direction~$iu$:

$$\left(f_t:(x,u)\in T^1 S\mapsto\exp_{(x,u)}(tiu)\right)_{t\geq 0}$$

Unfortunately, it does not induce an immersion on all of~$\Sigma_\eta$. We will modify this isotopy to make it computable and prove Theorem \ref{introtheoremA}. First we study~$f_t$ in local explicit models, then we will glue these local models. We do it in a flat model.

Let~$S'=\C$ be the flat plane,~$\gamma_1=\R\times\{0\}$ and~$\gamma_2=\{0\}\times\R$. Let~$\Gamma'_1=\{\gamma_1\}$ and~$\Gamma'_2=\{\gamma_1,\gamma_2\}$ represent respectively an edge and a crossing. We define~$g_t:(z,u)\mapsto(z+tiu,u)$ similarly to~$f_t$. Thus it is enough to study~$g_t$ in this model. Let~$\eta$ be an Eulerian coorientation of~$\Gamma'_i$ and construct~$\Sigma'\subset T^1S'$ in  the same way as we construct~$\Sigma_\eta$, for both alternating and non-alternating vertices.

\begin{lemma}\label{localmodel}
Let~$\mathcal{N}\subset\Sigma'$ be a tubular neighbourhood of~$\dd\Sigma'$. Then, for every~$T>0$, there is a small smoothing of~$\Sigma'$ such that for all~$t>T$,~$(\pi\circ g_t)_{|(\Sigma'\setminus\mathcal{N})}$ is an immersion.
\end{lemma}

\begin{proof}
We first prove the result for~$T$ arbitrary large and for a fixed smoothing of~$\Sigma'$. We first consider~$\Gamma'_1$. In this case we have~$\Sigma'=\{(x,e^{i\theta}), x\in\R, \epsilon\theta\in [0,\pi]\}$, where~$\epsilon = \pm1$ depends on the coorientation~$\eta$.
We have~$dg_t(x,\e^{i\theta})=dx-\epsilon t\e^{i\theta}d\theta$, which is injective if~$\theta\in(0,\pi)$. Thus~$g_t$ is an immersion on the interior of~$\Sigma'$.

Consider now~$\Gamma'_2$. Fix~$t>0$ and take~$(x,u)\in T^1S$. Then~$\ker(d(\pi\circ g_t)(x,u))$ is directed by~$U_t=(u,\frac{1}{t}\frac{\dd}{\dd\theta})$. But~$\lim\limits_{t \rightarrow +\infty} U_t = (u,0)$ which generates the geodesic flow. Let~$K\subset\Sigma'\setminus\dd\Sigma'$ be a compact sub-manifold. Then for~$t>0$ large enough,~$U_t$ is transverse to~$K$, so~$(\pi\circ g_t)_{|K}$ is an immersion. We can suppose that outside a compact~$K'\subset S'$,~$\Sigma'$ has been smoothed so that~$\Sigma'\setminus T^1_{K'}S'\subset \pi^{-1}(\Gamma'\setminus K')$. Then~$\Sigma'\setminus (T^1_{K'}S'\cup\dd\Sigma')$ is transverse to~$U_t$ for all~$t>0$, as in the first case.

We can combine these two transversal properties. Let~$\mathcal{N}\subset\Sigma'$ be a neighbourhood of~$\dd\Sigma'$. Let~$K'$ be a compact as above, and~$K=(T^1_{K'}S'\cap\Sigma')\setminus\mathcal{N}$. By what precedes, there exists~$T>0$ such that for all~$t>T$,~$(\pi\circ g_t)_{|(\Sigma'\setminus\mathcal{N})}$ is an immersion.

To prove that~$T$ can be made arbitrary small if we change the smoothing of~$\Sigma'$, it is enough to conjugate the previous isotopy with the diffeomorphism~$(z,u)\in T^1\C\mapsto(sz,u)$ for~$s$ a fixed parameter small enough. This diffeomorphism makes the smoothing of~$\Sigma'$ smaller and proves the lemma.
\end{proof}

\begin{proof}[Proof of Theorem \ref{introtheoremA}]
Let~$\mathcal{N}\subset\Sigma_\eta$ be a small tubular neighbourhood of~$\dd\Sigma_\eta$ that does not intersect the skeleton~$X$ of~$\Sigma_\eta$. Consider a flat metric~$\tilde g$ on a small neighbourhood of~$\Gamma\subset S$, such that~$\Gamma$ stays geodesic for~$\tilde g$. Then by compactness, there exists a finite open cover~$\mathcal U$ of a small neighbourhood of~$\Gamma\subset S$. For each~$U\in\mathcal U$, we can find an isometry between~$U$ and an open subset~$V$ of the standard model~$S'$ containing~$0$. Then by using Lemma \ref{localmodel}, we can find an isotopy of~$(\Sigma\subset\mathcal N)\cap T^1U$ to an immersion that stays in any small thickening of~$V$. Since the isotopies are parallel transports of the form~$(x,u)\mapsto(x + \lambda iu,u)$, with the metric~$\tilde g$, we can glue these isotopies for all~$U\in\mathcal U$.

Therefore there exists an isotopy of~$\Sigma_\eta\setminus \mathcal{N}$ whose composition with~$\pi:T^1S\to S$ ends with an immersion. We compose this isotopy with a retraction of~$\Sigma_\eta$ into a small neighbourhood of~$X\subset\Sigma_\eta\setminus\mathcal{N}$. The image of the neighbourhood of both kinds of vertices are obtained explicitly with the local model~$S'$.
\end{proof}

\begin{remarks}
We could have chosen to take~$t<0$. The image of the immersion would be similar. Also the order of the self-intersections of the immersion does not matter since it comes from the projection of a~$S^1$ fiber.
\end{remarks}

\begin{figure}[H]
	\centering
	\begin{center}
	\begin{picture}(65,160)(0,0)
	\put(0,0){\includegraphics[height=160mm]{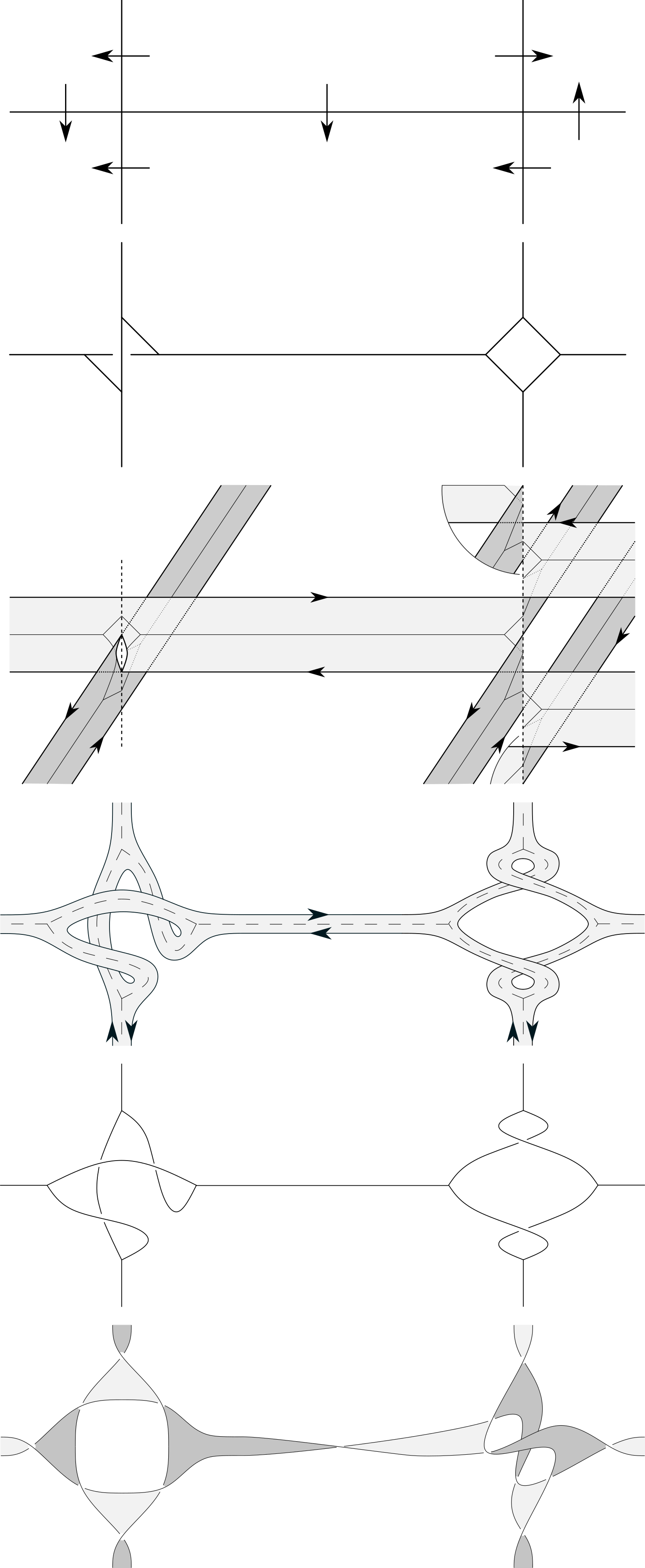}}
	\put(0,154){$a)$}
	\put(0,129){$b)$}
	\put(0,106){$c)$}
	\put(0,73){$d)$}
	\put(0,45){$e)$}
	\put(0,19){$f)$}
	\end{picture}
	\end{center}
	\caption{Local representations:~$a)$ a multi-geodesic~$\Gamma$ and a coorientation~$\eta$ of~$\Gamma$,~$b)$~the projection~$\pi(X)\subset S$ of the skeleton~$X$,~$c)$ a local picture of the surface~$\Sigma_\eta\subset T^1S$ locally identified with~$\R^2\times\R$,~$d)$ the isotope immersion~$\Sigma_\eta\to S$ provided by Theorem \ref{introtheoremA},~$e)$ the ribbon representation of the skeleton of~$\Sigma_\eta$ (see Section \ref{tubularNeighbourhood}) and~$f)$ the twisted representation of~$\Sigma_\eta$.} \label{fig:representations}
	\end{figure}
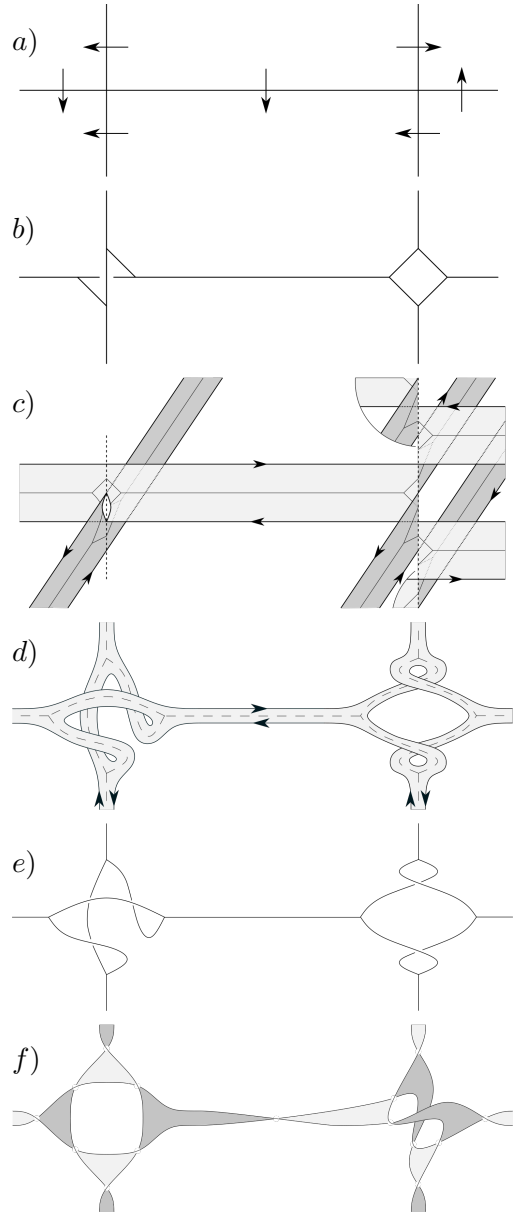

\begin{definition}\label{def:psi_im}
 The immersion~$\Sigma_\eta\to S$ thus constructed does not depend on the choices made (for~$t>0$) up to isotopy through immersion. It will be called the \textbf{isotope immersion}, and denoted by~$\psi_{im}$.
\end{definition}

\paragraph{The isotopy by twisted immersion.} There exists another representation of~$\Sigma_\eta$ that can be interesting. Take~$X$ the skeleton of~$\Sigma_\eta$ as in Figure \ref{fig:representations}.b. Replace every vertex of~$X$ (which has degree~$3$) by a triangle and replace every edge by a twisted rectangle. Glue them along the triangle corresponding to the ends of the edges. We obtain the image of a "twisted immersion" of~$\Sigma_\eta$ as in Figure \ref{fig:representations}.f. There exists a small isotopy of~$\Sigma_\eta$ in~$T^1S$ that gives this representation when composed with~$\pi$. One can prove this by using either the isotope immersion, or by understanding geometrically how to twist a rectangle~$r_e$ around a vertex~$v$ depending on the orientation of~$\eta$ around~$v$. However we will not use this representation later.

\subsection{The ribbon graph representation of~$\Sigma_\eta$.\label{tubularNeighbourhood}}

In this subsection, we adapt and use the notion of ribbon graph to describe~$\Sigma_\eta$ and its skeleton~$X$ as combinatorial objects. We start by defining the combinatorial tools that interest us. Then we connect them with the isotope immersion of~$\Sigma_\eta$. Eventually, this will make the study of isotopies and diffeomorphisms easier.

\begin{definition}
Let~$S$ be a surface,~$X=(X_0, X_1)$ a graph and~$\phi:X\to S$ a continuous map. We say that~$(X,\phi)$ is a \textbf{ribbon graph} if
\begin{itemize}
\item~$\phi_{|X_0}$ is injective,
\item for all~$e\in X_1$ (as closed segment),~$\phi_{|e}$ is immersed,
\item for all~$v\in X_0$, the tangents to~$\phi_{|e}$, for all~$e\in X_1$ bounding~$v$, are pairwise not positively collinear (and not zero).
\end{itemize}
\end{definition}

\begin{definition}
Let~$(X,\phi)$ be a ribbon graph on a surface~$S$. We call \textbf{induced surface} of~$(X,\phi)$ the thickened immersed surface obtained from the blackboard framing. More precisely, it corresponds to~$(\Sigma_{X,\phi},\iota,\pi)$, where~$\Sigma_{X,\phi}$ is a smooth surface,~$\iota:X\looparrowright \inte(\Sigma_{X,\phi})$ is an embedding and~$\pi:\Sigma_{X,\phi}\hookrightarrow S$ is an immersion, such that~$\Sigma_{X,\phi}$ retracts by deformation to~$\iota(X)$, and~$\pi\circ\iota = \phi$.
\end{definition}

\begin{example}
The isotope immersion defined in Section \ref{smallIsotopy} naturally yields a ribbon graph and its induced surface, as in Figure \ref{fig:representations}.
\end{example}

\begin{definition}
Two ribbon graphs are said to be \textbf{weakly isotopic} if there exists a succession of isotopies of ribbon graphs, of twists, fusions and contractions that goes from one to the other. The twist, fusion and contraction moves are represented in Figure~\ref{TwistMove}.
\end{definition}

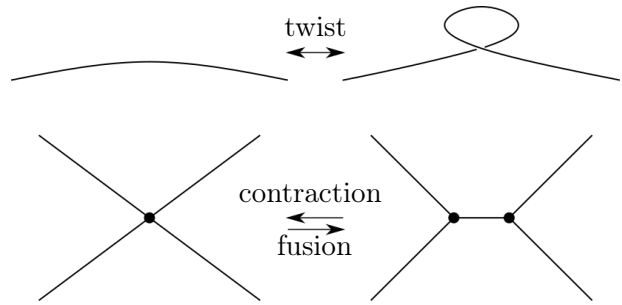
\begin{figure}[ht]
\centering
\begin{center}
\begin{picture}(80,38)(0,0)
\put(0,0){\includegraphics[width = 80mm]{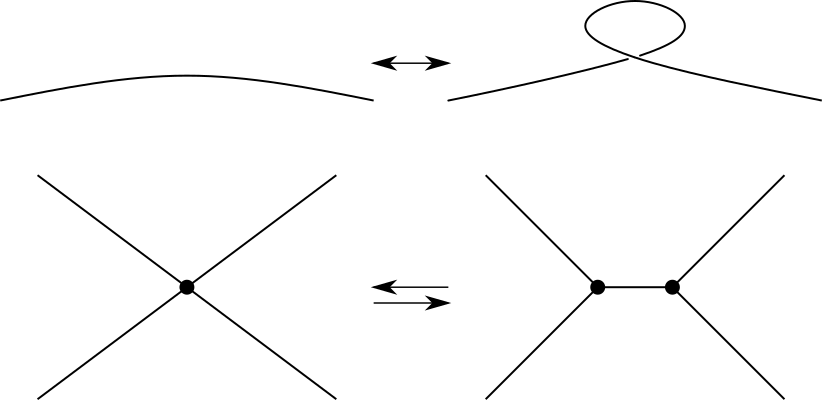}}

\put(36, 35){twist}
\put(30, 13){contraction}
\put(35, 6){fusion}
\end{picture}
\end{center}
\caption{Weak isotopy of ribbon graph.}
\label{TwistMove}
\end{figure}

Notice that during such an isotopy, the order of the edges around a vertex do not change. Examples of weak isopoties are given in Figure \ref{fig:slide}.

\begin{proposition}
Let~$(X,\phi)$ and~$(Y,\psi)$ be two weakly isotopic ribbon graphs. Then their induced surfaces are diffeomorphic. Also this diffeomorphism together with the retractions~$\Sigma_{X,\phi}\to X$ and~$\Sigma_{Y,\psi}\to Y$ induce the same homotopy equivalence~$X\simeq Y$ than the weak isotopy.
\end{proposition}

Because of the twist move, this diffeomorphism does not always comes from an isotopy of their immersed image in~$S$.

\paragraph{The combinatorial representation of~$\Sigma_\eta$ \label{comparation}.} Theorem \ref{introtheoremA} gives a representation of~$\Sigma_\eta$ as a ribbon graph. In this paragraph, we use this representation to compare the alternating and non-alternating vertices, that can be found in Figure \ref{fig:representations}.e. This will be useful in Section \ref{section:combiflip} for identifying the surface~$\Sigma_\eta$ one with another.

We will detail here how to compare two Birkhoff sections associated to two coorientations that differ around a specific vertex~$v$ of~$\Gamma$. Figure \ref{fig:slide} describes two isotopies of ribbon graphs that interest us. The idea of the isotopies is the following. Use the image~$\pi(X)\subset S$ of the skeleton~$X$. There is in~$\pi(X)$ a (maybe twisted) square~$\diamond$ associated to the vertex~$v$. Fix~$e$ a non twisted edge of~$\diamond$. In the square~$\diamond$, the edge~$e$ has two adjacent neighboring edges, that we move along~$e$.

\begin{definition}\label{def:slide}
The isotopies described above and shown in Figure \ref{fig:slide} are called \textbf{slide along}~$e$.
\end{definition}

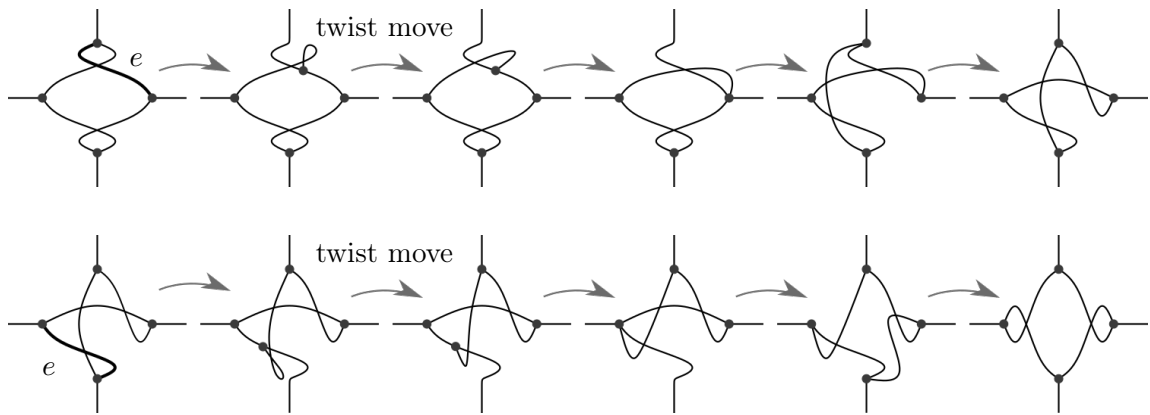
\begin{figure}[ht]
\centering
\begin{center}
\begin{picture}(150,55)(0,0)
\put(0, 30){\includegraphics[width = 150mm]{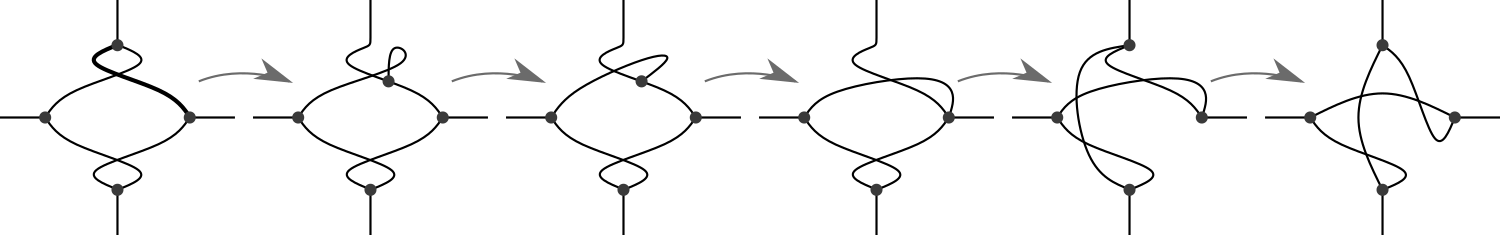}}

\put(16, 46){$e$}
\put(40.5, 50){twist move}

\put(0,0){\includegraphics[width = 150mm]{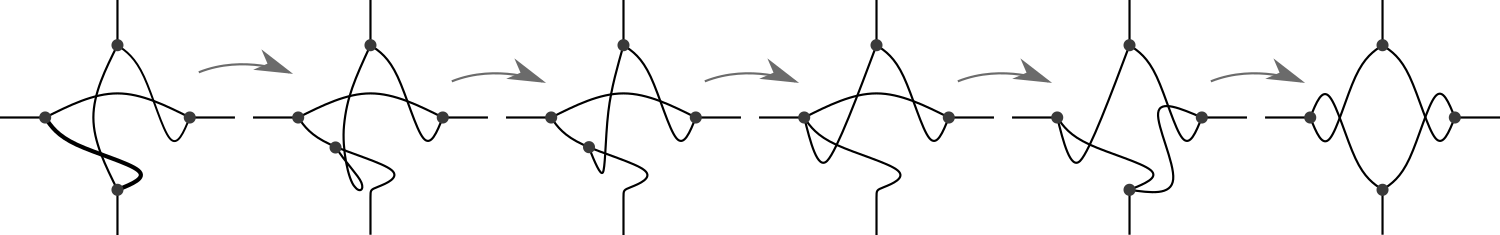}}

\put(4.5, 5){$e$}
\put(40.5, 20){twist move}
\end{picture}
\end{center}
\caption{Isotopy of ribbon graphs (slide along~$e$)}
\label{fig:slide}
\end{figure}

In Figure \ref{fig:slide} are represented respectively a slide along the top right edge and along the bottom left edge. All possible slides are obtained by doing rotation or symmetry of these two slides. Notice that these slides are "pseudo involutions", in the sense that sliding along the same edge twice is isotopic to the identity. We are interested in compositions of slides.

Let~$v$ in~$\Gamma_0$ and denote by~$c_1,c_2,c_3,c_4$ the four quadrants around~$v$, ordered according to an Eulerian coorientation, for example as in Figure \ref{fig:slidesordering}. Denote by~$sl_i$ the slide along~$c_i$, which is well-defined when the skeleton~$X$ admits an edge~$e_i$ along~$c_i$.

\begin{figure}[ht]
\centering
\begin{center}
\begin{picture}(80,45)(0,0)
\put(0,5){\includegraphics[width = 80mm]{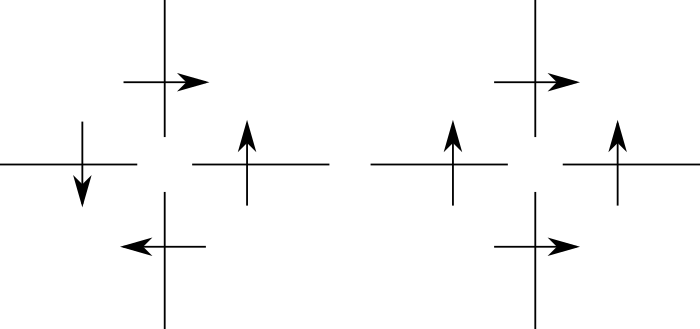}}
\put(10,0){alternating}
\put(47,0){non-alternating}
\put(17.7,23){$v$}
\put(60.2,23){$v$}

\put(27,31.7){$1$}
\put(8.5,31.7){$2$}
\put(27,13.2){$3$}
\put(8.5,13.2){$4$}

\put(69.5,31.7){$1$}
\put(51,31.7){$2$}
\put(69.5,13.2){$3$}
\put(51,13.2){$4$}
\end{picture}
\end{center}
\caption{Ordering the slides around a vertex~$v$.}
\label{fig:slidesordering}
\end{figure}
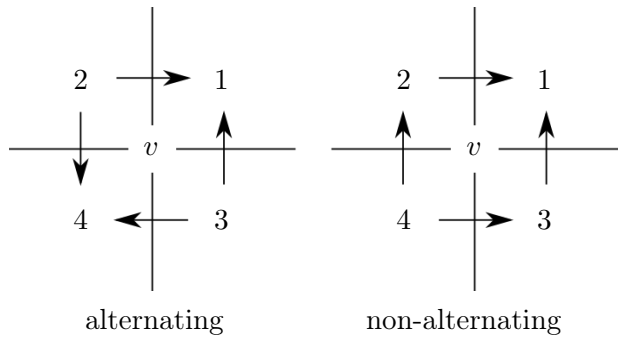

\begin{lemma}\label{lemma:DehnTwistVertex}
In the above context, the diffeomorphism of~$\Sigma_\eta$ induced by~$sl_4\circ sl_3\circ sl_2\circ sl_1$ is well-defined and isotopic to a negative Dehn twist along the curve~$\gamma_v$, represented in Figure \ref{fig:gammavf}.
\end{lemma}

\begin{remark}
The curve~$\gamma_v$ is a skeleton of~$\Sigma_\eta$ when restricted to a small neighbourhood of~$T^1_vS$.
\end{remark}

\begin{proof}
We prove the lemma when~$v$ is an alternating vertex. The other case only needs an adaptation of the diagram we will use. Let~$U\subset T^1S$ be a small tubular neighbourhood of the fiber~$T^1_vS$, so that~$U\cap\Sigma_\eta$ is homeomorphic to an annulus. Let~$\delta\subset\Sigma$ be a curve intersecting the core of~$U\cap\Sigma_\eta$ once, and with ends outside~$U$, as in Figure~$\ref{svertexisotopy1}$. Denote by~$f:\Sigma_\eta\to\Sigma_\eta$ the diffeomorphism induced by the isotopy~$sl_4\circ sl_3\circ sl_2\circ sl_1$.

In Figure \ref{svertexisotopy1}, we give the diagrams of four isotopies of ribbon graphs, and we keep track of~$\delta$ along these isotopies. It proves that the concatenation is well-defined, and that, in homology,~$f_\star([\delta])=[\delta]\pm[\gamma_v]$. Also the isotopy fixes the ribbon graph outside~$U$. So the support of~$f$ is included in an annulus, and~$f$ acts in homology like a Dehn twist. Figure \ref{svertexisotopy1} gives the sign of the Dehn twist. Thus it is isotopic to the negative Dehn twist along~$\gamma_v$.

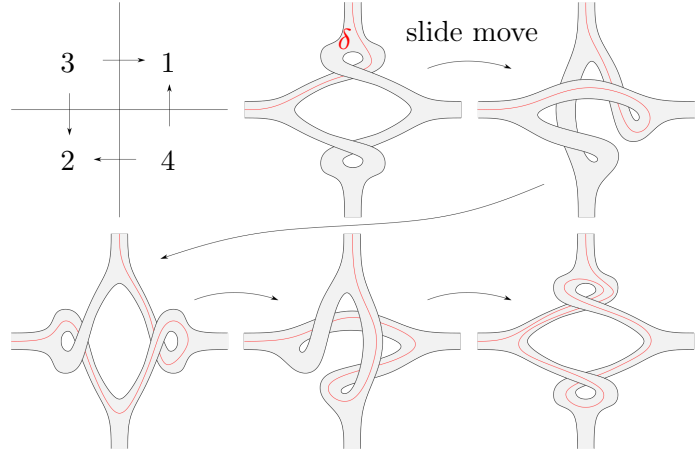
\begin{figure}[ht]
\centering
\begin{center}
\begin{picture}(90, 59)(0,0)
\put(0,0){\includegraphics[width = 90mm]{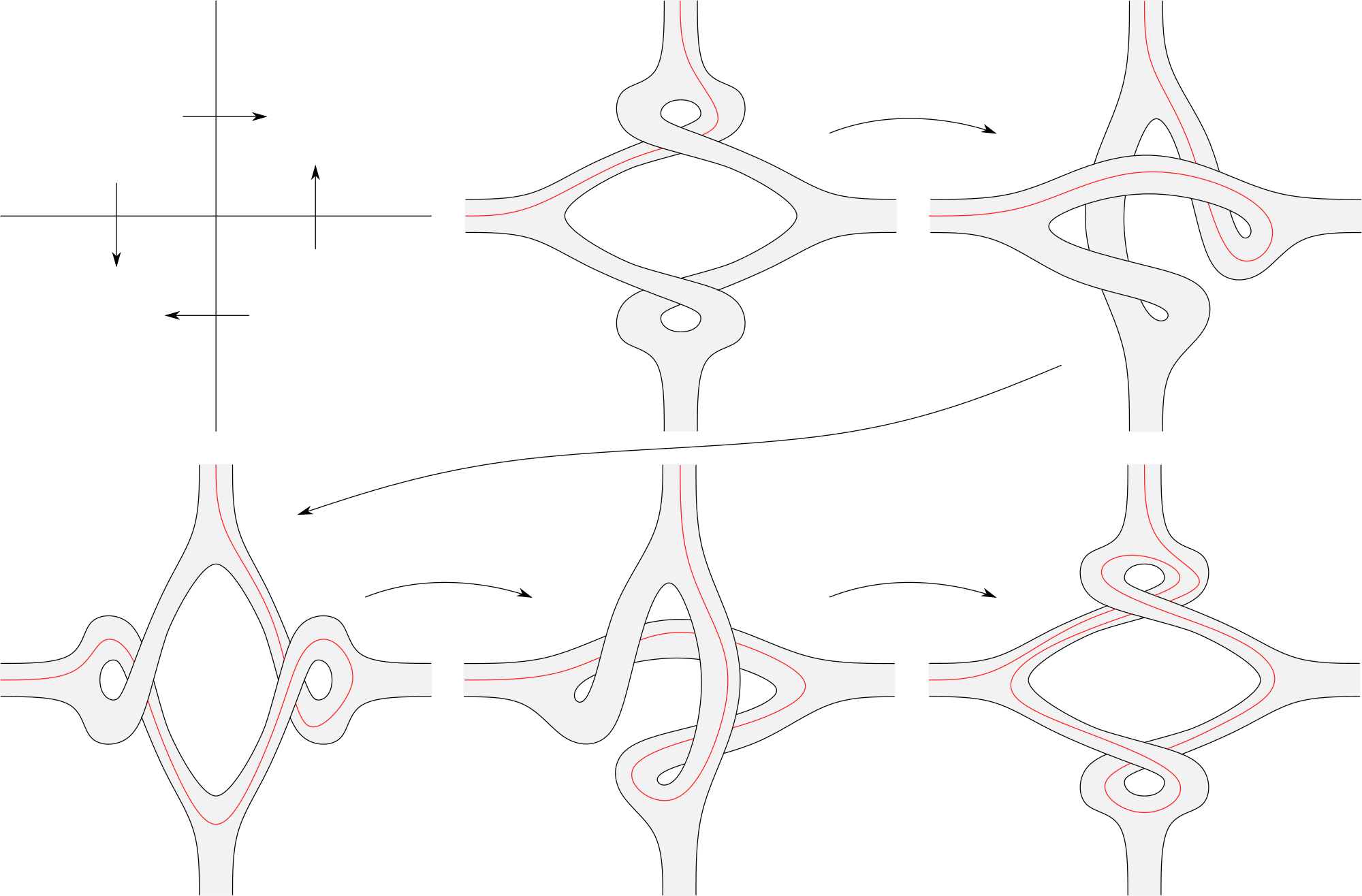}}

\put(52, 54){slide move}
\put(19.7, 50){$1$}
\put(6.5, 37){$2$}
\put(6.5, 50){$3$}
\put(19.7, 37){$4$}
\color{red}
\put(43, 53){$\delta$}
\color{black}
\end{picture}
\end{center}
\caption{Action of four slides around an alternating vertex~$v$ and their traces on~$\delta$.}
\label{svertexisotopy1}
\end{figure}

\end{proof}

\section{Elementary flips and partial return maps\label{section:combiflip}}

The main idea for computing the first-return map is to see it as a composition of partial return maps~$\Sigma_\eta=\Sigma_0\to\Sigma_1\to\hdots\to\Sigma_n=\Sigma_\eta$. In this section, we study the combinatorics and the geometry of the partial return maps, in order to prove Theorem~\ref{introtheoremB}. We also introduce tools needed to formulate Theorem \ref{introtheoremCrestated} precisely.

\subsection{Combinatorial flip transformation}

We introduce in this subsection the main combinatorial tool: the flip. We start by studying~$\Gamma^\star$ the dual graph of~$\Gamma\subset S$. In~$\Gamma^\star$, every face~$f$ of~$S\setminus\Gamma$ (diffeomorphic to~$B^2$) is replaced by a vertex~$f^\star$ inside the face. Every edge of~$e\in\Gamma_1$ between two faces~$f_1$ and~$f_2$ (not necessarily different) is replaced by a transverse edge~$e^\star$ from~$f_1^\star$ to~$f_2^\star$. And every vertex~$v\in\Gamma_0$ is replaced by a face~$v^\star$.

Let~$\eta$ be a coorientation of~$\Gamma$. It naturally induces an orientation on~$\Gamma^\star$, which will also be denoted by~$\eta$. We are interested by geodesics in~$S$ that induce on~$(\Gamma^\star,\eta)$ an oriented cycle. For a geodesic~$\gamma\subset\Gamma$, pushing slightly~$\gamma$, to its left or its right, induces two different cycles in~$\Gamma^\star$, but they are simultaneously oriented or not-oriented for~$\eta$ (for homology reasons). We consider these curves for telling whether~$\gamma$ induces an oriented cycle in~$(\Gamma,\eta)$.

\begin{lemma}\label{acyclic}
Let~$\eta$ in~$\coorientation{\Gamma}$, then:
\begin{itemize}
	\item For any curve~$\gamma\subset S$ inducing an oriented cycle in~$(\Gamma^\star,\eta)$, the geodesic homotopic to~$\gamma$ also induces an oriented cycle in~$(\Gamma^\star,\eta)$.
	\item The surface~$\Sigma_\eta$ is a Birkhoff section if and only if the oriented graph~$(\Gamma^\star,\eta)$ has no oriented cycle. In this case, we say that~$\eta$ is an \textbf{acyclic} coorientation.
	\item If~$\Gamma$ admits an acyclic coorientation, then every edge in~$\Gamma_1$ bounds two different faces of~$\Gamma$.
\end{itemize}
\end{lemma}

\begin{proof}

Let~$\gamma\subset S$ be a curve inducing an oriented cycle in~$\Gamma^\star$. Denote by~$\tilde\gamma$ the unique geodesic of~$S$ homotopic to~$\gamma$. We will prove that~$\tilde{\gamma}$ induces an oriented cycle in~$\Gamma^\star$ by doing Reidemeister move on~$\gamma$. Suppose that~$\tilde\gamma$ is not a component of~$\Gamma$, then~$\tilde\gamma$ is obtained from~$\gamma$ by doing Reidemeister moves on~$\Gamma\cup \gamma$, shortening~$\gamma$ and not changing~$\Gamma$. 

The curve~$\gamma$ induces an oriented cycle, so~$\eta(\gamma)=\pm|\Gamma\cap \gamma|$. If~$\delta$ is homotopic to~$\gamma$, then~$|\eta(\gamma)|=|\eta(\delta)|\leq |\Gamma\cap \delta|$, so~$\gamma$ minimises~$|\Gamma\cap \gamma|$ in its homotopy class. Hence~$\Gamma\cup \gamma$ has no bigon. Also, up to homotopy preserving~$\Gamma\cap\gamma$,~$\gamma$ can be taken without~$1$-gon. Hence~$\Gamma\cup \gamma$ has no~$1$-gon nor bigon, so no Reidemeister moves I and II can be applied without making~$\gamma$ longer. A Reidemeister III move on~$\Gamma\cap\gamma$, that do not change~$\Gamma$, changes the cycle in~$\Gamma^\star$ induced by~$\gamma$ only if it is along one arc of~$\gamma$ and two intersecting arcs of~$\Gamma$. Also since~$\gamma$ induces an oriented cycle of~$\eta$, such a Reidemeister III move must be in a neighbourhood of a non-alternating vertex, and after the move,~$\gamma$ still induces an oriented cycle on~$S$. Thus the geodesic~$\tilde\gamma$ induces an oriented cycle.

If~$\tilde\gamma$ is a component of~$\Gamma$, we can apply the same idea and prove that a slight push of~$\tilde\gamma$ on its right (or on it left) induces an oriented cycle.

We now prove the equivalence in the second point. Suppose that~$\Sigma_\eta$ is not a Birkhoff section. Then for arbitrary large~$T>0$, there exists~$(x,u)\in T^1S$ such that for~$\forall 0\leq t\leq T$,~$\phi_t(x,u)\not\in\Sigma_\eta$. Take~$T > n d$ where~$n=|\Gamma^\star|$ and~$d$ is the largest diameter of a face~$f\in\Gamma^\star$. Then the geodesic arc~$\phi_{[0,T]}(x,u)$ must travel through at least~$n+1$ faces (counted with multiplicity). Thus it induces in~$\Gamma^\star$ a path~$\gamma^\star$ that admits self-intersections. Note that the orientation of~$\gamma^\star$ in~$\Gamma^\star$ is the opposite to the one provided by~$\eta$. Hence a restriction of~$\gamma^\star$ between two self-intersections, with the opposite orientation, is an oriented cycle in~$(\Gamma^\star, \eta)$.

Suppose that there is an oriented cycle in~$(\Gamma^\star,\eta)$. By the first point, there exists a closed geodesic~$\gamma$ inducing an oriented cycle. If~$\gamma\not\subset\Gamma$, then the orbit~$\larrow{\gamma}$ of the geodesic flow given by the geodesic~$\gamma$ lifted with the opposite direction, satisfies~$\larrow{\gamma}\cap \Sigma_\eta=\emptyset$. Then~$\Sigma_\eta$ is not a Birkhoff section. Suppose that~$\gamma\subset\Gamma$, and~$\larrow{\gamma}\subset\partial\Sigma_\eta$. Then every orbit in the stable leaf of~$\larrow{\gamma}$ stops intersecting~$\Sigma_\eta$ after a large enough time, since any slight push of~$\gamma$ in the appropriate direction induces an oriented cycle of~$\eta$. Hence in both case~$\Sigma_\eta$ is not a Birkhoff section.

For the last statement, it is enough to notice that an edge in~$\Gamma$ bounded twice by the same face is dual to a loop in~$\Gamma^\star$.
\end{proof}

When~$\Sigma_\eta$ is a Birkhoff section,~$(\Gamma^\star, \eta)$ is acyclic and~$\eta$ induces an order on the finit set~$\Gamma^\star$. Thus~$\eta$ must have at least one \textbf{sink} face, that is,~$\eta$ is going inward~$f$ as in Figure \ref{fig:flip}.

\begin{definition}
Let~$\eta$ in~$\coorientation{\Gamma}$ and~$f$ be a sink face. We define~$I_f(\eta)\in\coorientation{\Gamma}$ the coorientation obtained by flipping~$\eta$ along~$\partial f$. We call~$I_f$ an \textbf{elementary flip} along~$f$. We also define recursively~$I_{\{f_1,\hdots,f_k\}}(\eta)=I_{f_k}(I_{(f_1,\hdots,f_{k-1})}(\eta))$, when recursively~$f_i$ is a sink face of~$I_{(f_1,\hdots,f_{i-1})}(\eta)$ for all~$1\leq i\leq k$.
\end{definition}

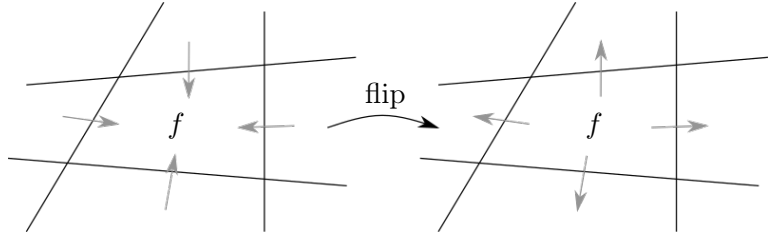
\begin{figure}[ht]
\centering
\begin{center}
\begin{picture}(100,31)(0,0)
\put(0,0){\includegraphics[width = 100mm]{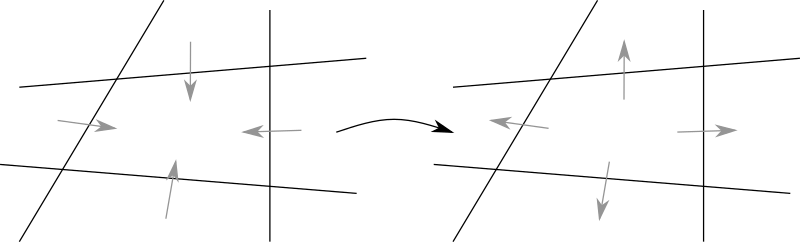}}
\put(21,13){$f$}
\put(76,13){$f$}
\put(47,17){flip}
\end{picture}
\caption{A sink face on the left, and a flip.}
\label{fig:flip}
\end{center}
\end{figure}

If~$\eta$ is Eulerian,~$I_f(\eta)$ remains Eulerian and is cohomologous to~$\eta$.

\paragraph{Representations.}

Around a vertex~$v$, an Eulerian coorientation of~$\Gamma$ gives a local ordering on the~$4$ adjacent faces (so that the coorientation is decreasing). We extend the ordering, by ordering~$v$ relatively to these faces using Figure~\ref{fig:augmented}. That is, if~$v$ is alternating, we set~$v$ bigger than the sink faces and smaller than the source faces. If~$v$ is not alternating, we set~$v$ smaller than the source face and bigger than the three other faces. We call this ordering on~$\Gamma_2\cup\Gamma_0$ the \textbf{coherent order}. These orderings represent the order of the Dehn twists in the product in Theorem \ref{introtheoremCrestated}. 

\begin{remark}
Suppose that~$\Sigma_\eta$ is a Birkhoff section. If one face~$f$ covers two quadrants around a vertex~$v$, then by Lemma \ref{acyclic} it must be two opposite quadrants. Also Lemma \ref{acyclic} prevents~$f$ to be the sink and the source quadrants of a non-alternating vertex~$v$. In this case, the coherent ordering is still well-defined on~$\Gamma_2\cup\Gamma_0$. 

If there exist two faces such that both of them cover two opposite quadrants around~$v$, the coherent ordering is still well-defined on~$\Gamma_2\cup\Gamma_0$ for the same reasons. 
\end{remark}

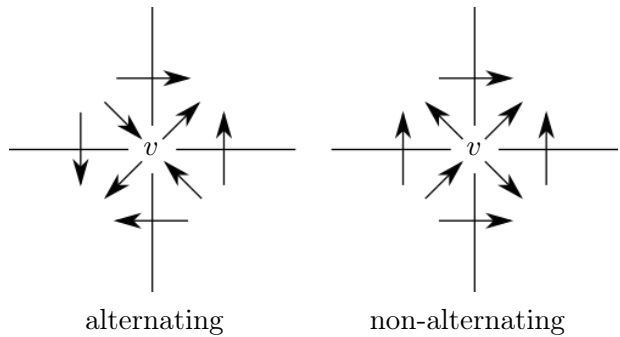
\begin{figure}[ht]
\centering
\begin{center}
\begin{picture}(80,45)(0,0)
\put(0,5){\includegraphics[width = 80mm]{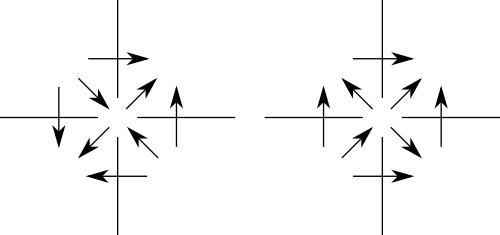}}
\put(10,0){alternating}
\put(47.5,0){non-alternating}
\put(17.7,23){$v$}
\put(60.2,23){$v$}
\end{picture}
\end{center}
\caption{Coherent ordering of a vertex relatively to its adjacent faces.}
\label{fig:augmented}
\end{figure}

\begin{definition}
Let~$\eta$ be an acyclic Eulerian coorientation of~$\Gamma$. We call a \textbf{partial representation} of~$\eta$ a total order on~$\Gamma_2$, which extends the coorientation~$\eta$. We call a \textbf{representation} of~$\eta$ a total order on~$\Gamma_0\cup\Gamma_2$ which extends the coorientation~$\eta$ and the coherent order. Thanks to acyclicity, representations always exist.
\end{definition}

\begin{example}
If~$\Gamma\equiv 0\in H_1(S, \sfrac{\Z}{2\Z})$, then the faces of~$\Gamma$ can be colored in black and white, and we can take the Eulerian coorientation~$\eta$ that goes from black to white. Then a representation can look like: \textit{white faces totally ordered~$<$ vertices totally ordered~$<$ black faces totally ordered}. This choice of representation will leads to the composition of three Dehn multi-twists studied by N.A'Campo and M.Ishikawa.
\end{example}

The point is to use and deform a representation and its coorientation in order to represent the first-return map as a product of elementary diffeomorphisms. We have defined an elementary operation on coorientations, that we will extend to representations.

\begin{definition}
Let~$(\eta,\leq)$ be an acyclic Eulerian coorientation with one partial representation. We define~$I(\eta,\leq)$ to be~$(I_f(\eta),I_f(\leq))$, where~$f$ is the minimal face of~$\leq$ and~$I_f(\leq)$ is obtained from~$\leq$ by setting~$f$ to the maximum. It is called the \textbf{elementary flip} of~$(\eta,\leq)$.
\end{definition}

\subsection{Algorithm for the first-return map.} 

In order to describe the first-return map, we will first describe how it acts on the representations of acyclic Eulerian coorientation. Let~$\eta$ be such a coorientation and~$\leq$ one partial representation. By iterating the flip~$I$, we create a family of~$\#\Gamma_2$ coorientations and partial representations, before looping to~$(\eta,\leq)$. We will translate this geometrically later. For now let us detail a bit more what the coorientations obtained in this process look like.

\begin{lemma}\label{lemma:partialreturns}
Let~$(\eta,\leq)$ be a partial representation. Let~$1\leq k<n$,~$f$ be the~$k^{th}$ face for~$\leq$ and~$(\nu,\preceq)=I^k(\eta,\leq)$. For every~$e\in\Gamma_1$ bounded by two faces~$f_1$ and~$f_2$, we have~$\nu(e)=\eta(e)$ if and only if~$f_1$ and~$f_2$ are simultaneously greater than~$f$ for~$\leq$, that is, either ($f_1>f$ and~$f_2>f$) or ($f_1\leq f$ and~$f_2\leq f$).

In particular~$I^n(\eta,\leq)=(\eta,\leq)$.
\end{lemma}

\begin{proof}
The partial representation~$\preceq$ differ from~$\leq$ by moving the~$k$ lower faces on top. So we have~$\nu(e)\neq\eta(e)$ if and only if one of the~$f_i$ is in this subset, and the other is not.
\end{proof}

This lemma will be needed in the next section. The algorithm that consists in applying elementary flips~$I_f$ for successive minimal faces~$f$ will be called by the \textbf{flip algorithm}. This algorithm gives a way to compute the first-return map by computing the~$n$ elementary flips that correspond to the iteration of~$I$.

\subsection{Partial return map\label{subsection:partialReturnMap}}

The partial return maps are the geometric realisation of the combinatorial flip. We define the partial return maps and prove Theorem \ref{introtheoremB} in this subsection.

Let~$\eta$ in~$\coorientation{\Gamma}$ and~$f\in\Gamma^\star_0$ a sink face for~$\eta$. Write~$\Sigma_1=\Sigma_\eta$ and~$\Sigma_2=\Sigma_{I_f(\eta)}$. The elementary flip~$I_f$ acts geometrically by pushing~$\Sigma_1$ along the geodesic flow only around the face~$f$, as in Figure \ref{fig:flipdepth}. Define~$h:\Sigma_1\to\R^+$ such that~$h(x)$ is the smallest~$t\geq 0$ such that~$\phi_t(x,u)$ in~$ \Sigma_2$, and~$r_f:\Sigma_1\to \Sigma_2$ by~$r_f(x,u)= \phi_{h(x)}(x,u)$.

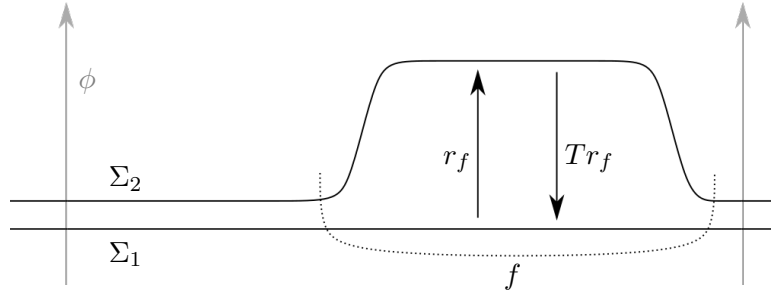
\begin{figure}[ht]
\centering
\begin{center}
\begin{picture}(100, 38)(0,0)
\put(0,0){\includegraphics[width = 100mm]{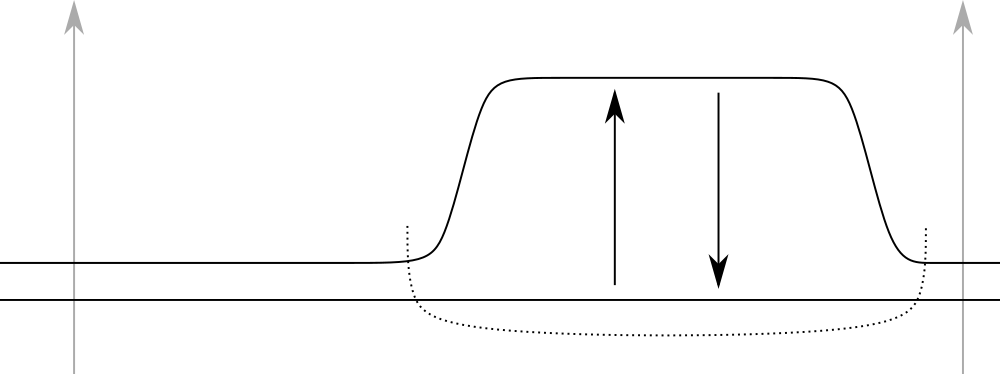}}

\put(57, 16){$r_f$}
\put(73, 16){$T{r_f}$}
\put(65, 0){$f$}
\put(13, 3){$\Sigma_1$}
\put(13, 13){$\Sigma_2$}

\color{gray}
\put(9, 26){$\phi$}
\color{black}

\end{picture}
\end{center}
\caption{Relative positions of~$\Sigma_1 = \Sigma_\eta$ and~$\Sigma_2 = \Sigma_{I_f(\eta)}$ inside~$T^1S$.}
\label{fig:flipdepth}
\end{figure}

\begin{proposition}
There exist two smoothings of~$\Sigma_1$ and~$\Sigma_2$,~$\epsilon>0$ arbitrary small and~$U$ the complement of a small neighbourhood of~$T^1_fS$ such that :
\begin{itemize}
\item~$\Sigma_1$ and~$\Sigma_2$ are disjoint and~$r_f:\Sigma_1\to\Sigma_2$ is well-defined and smooth,
\item~$\phi_{-\epsilon}(\Sigma_2)\cap U = \Sigma_1\cap U$ and~$(\phi_{-\epsilon}\circ r_f)_{|\Sigma_1\cap U} = \Id$.
\end{itemize}
\end{proposition}

We call~$r_f$ a \textbf{partial return map}. It does not depend on the smoothing of~$\Sigma_1$ and~$\Sigma_2$, so we can define it without precision on the smoothing.

\begin{proof}
We write~$\ehat{\Sigma_i}$ for the~$2$-complex that we smooth for constructing~$\Sigma_i$ (without its boundary). First define~$\ehat{h}:\ehat\Sigma_1\to\R$ and~$\ehat{r}:\ehat{\Sigma_1}\to \ehat{\Sigma_2}$ in the following way. Let~$(x,u)$ in~$\ehat{\Sigma_1}$ and not in~$T^1_v\cap\ehat{\Sigma_1}$ for any non-alternating vertex~$v$. If~$x$ is in~$f$ and~$u$ goes inside~$f$, define~$\ehat{r}(x,u)$ to be the first intersection of~$\ehat{\Sigma_2}$ and of the geodesic starting at~$(x,u)$, and~$\ehat{h}(x,u)$ to be the length of this geodesic arc. Elsewhere set~$\ehat{r}(x,u)=(x,u)$ and~$\ehat{h}(x,u)=0$. 

Let~$v$ be a non-alternating vertex and take~$(x,u)\in T^1_v\cap\ehat{\Sigma_1}$. After the desingularisation of~$\ehat{\Sigma_1}$, two points of~$\ehat{\Sigma_1}$ correspond to~$(x,u)$ and we must define~$\hat{r}$ and~$\hat{h}$ for both points. One of them is adjacent to the two edges of~$f$ adjacent to~$v$, and we define~$\ehat{h}$ and~$\hat{r}$ on it as if it was going inside~$f$. The other is adjacent to the two other edges, and we define~$\ehat{h}$ and~$\hat{r}$ on it as it was outside of~$f$. 

Both functions~$\hat{h}$ and~$\hat{r}$ are well-defined and continuous. We smooth together~$\ehat\Sigma_1$,~$\ehat\Sigma_2$,~$\ehat{h}$ and~$\ehat{r}$ into~$\Sigma_1$,~$\Sigma_2$,~$h$ and~$r$. We use smoothings smaller that~$\epsilon/3$.

On a small neighbourhood of each corner of~$f$,~$h$ may be negative. To make~$h$ positive, take~$g$ a negative smoothing of~$-\max(0,-h)$ and push~$\Sigma_1$ with~$\phi_g$. We suppose that~$|g + \max(0,-h)| < \epsilon$ and that~$g = -\max(0,-h)$ outside the tubular neighbourhood~$B(f,\epsilon)$ of~$f$. Now~$h\circ\phi_{-g}\geq 0$.

Let~$U=T^1S\setminus T^1_{B(f,\epsilon)}S$ be the complement of~$f$ in~$T^1S$. By construction~$\Sigma_1\cap U = \Sigma_2\cap U$ and~$(r_f)_{|\Sigma_1\cap U} = \Id$. We finish by replacing~$\Sigma_2$ by~$\phi_\epsilon(\Sigma_2)$.
\end{proof}

Fix a representation of~$\eta$. The flip algorithm generates a family of cohomologous Birkhoff sections, consecutively disjoint. The partial return maps describe how the flow moves one to the next one.

\setcounter{introtheorem}{1}
\begin{introtheorem}
Let~$\Gamma\subset S$ be a filling geodesic multi-curve of a hyperbolic surface,~$\eta\in\coorientation{\Gamma}$ acyclic,~$\leq$ a partial representation of~$\eta$ and denote the faces by~$f_1\leq\cdots\leq f_n$. Denote by~$\Sigma_0=\Sigma_\eta$ and successively the partial return map~$r_i:\Sigma_{i-1}\to\Sigma_i$ the partial return map along the face~$f_i$. Then~$\Sigma_n=\Sigma_\eta$ and the first-return map on~$\Sigma_\eta$ is the product of the partial return maps~$r_{\Sigma_\eta}=r_n\circ\cdots\circ r_1$.
\end{introtheorem}

Equivalently, the family of cohomologous Birkhoff sections are pairwise disjoint.

\begin{proof}[Proof of Theorem \ref{introtheoremB}]
We will prove that~$r=r_{\Sigma_\eta}= r_n\circ\cdots\circ r_1$ on a dense subset of~$\Sigma_0$.  Let~$(x,u)$ be in~$\inte(\Sigma_\eta)$ such that the geodesic starting at~$(x,u)$ intersects~$\Sigma_\eta$ again before intersecting~$T^1_{\Gamma_0}S$. This represents a dense subset of~$\Sigma_\eta$. We can suppose that the smoothings have been done away from the short geodesic starting at~$(x,u)$ and ending on~$\Sigma_0$ when it first intersects it. So for~$U$ a small neighbourhood of the geodesic from~$(x,u)$ to~$r(x,u)$, we have~$\Sigma_i\cap U=\ehat\Sigma_i\cap U$. Denote by~$f$ the face~$f_j$ at which~$u$ is going inside. By definition of the partial return maps, we have~$r_k\circ\hdots\circ r_1(x,u)=\phi_{t_h}(x,u)\in\Sigma_k$ for some~$t\geq 0$. But~$t_j>0$ since~$(x,u)\in T^1_fS$ and going inside~$f$. Also~$(t_k)_k$ is increasing, so~$t=t_n>0$. It remains to prove that~$t$ is the minimal~$s>0$ that satisfies~$\phi_s(x,u)\in\Sigma_n$.

The main idea is that~$r_k\circ\hdots\circ r_1(x,u)$ remains constant in~$k$, once it is in~$\Sigma_n$ for one~$k\geq j$. Therefore we would have~$r_n\circ\hdots\circ r_1(x,u)=r_k\circ\hdots\circ r_1(x,u)$ for the minimal~$k\geq j$ that satisfies~$r_k\circ\hdots\circ r_1(x,u)\in\Sigma_n$, and~$t>0$ would be the minimal one.

To prove this, let~$j\leq k \leq n$ and suppose~$r_k\circ\hdots\circ r_1(x,u)=(y,v)\in\Sigma_0$. Denote by~$f_p$ the face in which~$v$ is pointing,~$f_q$ the face in which~$v$ is going out, and~$e\in\Gamma_1$ the edge containing~$y$. We claim that~$q\leq k$, because since~$k\geq j$,~$(y,v)$ is going out of the last face that affected the product~$r_k\circ\hdots\circ r_1(x,u)$. Since~$(y,u)\in\Sigma_k\cap\Sigma_0$, we have~$\eta_k(e)=\eta_0(e)$. Then Lemma \ref{lemma:partialreturns} implies that~$p\leq k$. Thus for all~$k< l\leq n$,~$l>p$ so~$(y,v)\not\in\supp(r_l)$, and by induction~$r_l\circ\hdots\circ r_1(x,u)=(y,v)$.
\end{proof}

\section{Explicit first-return map\label{section:first_return_map}}

Let~$r:\Sigma_\eta\to\Sigma_\eta$ the first return map for the geodesic flow. In Section \ref{section:combiflip}, we have decomposed the first return map as a product of partial return maps. In this section, we first compare these partial return maps to negative Dehn twists, along prescribed curves. Then we state and prove Theorem~\ref{introtheoremCrestated}. We finish by comparing several decompositions of first return maps in Dehn twists, and prove Corollary \ref{introcorollaryD}.

\subsection{Explicit computation of partial return maps.}

Let~$f$ be a sink face of~$\eta\in\coorientation{\Gamma}$ and denote~$\Sigma_1=\Sigma_\eta$ and~$\Sigma_2=\Sigma_{I_f(\eta)}$. We will compare the partial return map~$r:\Sigma_1\to\Sigma_2$ to a negative Dehn twist, but~$r$ is not an endomorphism. We need to correct it with a simple diffeomorphism~$\cor:\Sigma_2\to\Sigma_1$ so that~$\Sigma_1\xrightarrow{r}\Sigma_2\xrightarrow{\cor}\Sigma_1$ can be expressed as a Dehn twist. In order to find~$\cor$, we use the ribbon representation of~$\Sigma_i$ and the slides from Definition \ref{def:slide}.

To simplify the computation of~$\cor\circ r$, we need to precise which slides we use. Let~$\{c_1,\hdots,c_k\}$ be the set of corners of~$f$. If~$f$ has double corners, we consider them twice. For~$1\leq i\leq k$, the ribbon graph of~$\Sigma_2$ around~$c_i$ as an edge corresponding to the vector based on~$c_i$ and going inside~$f$. We denote by~$e_i$ this edge, as in Figure \ref{fig:correctioncorner}.

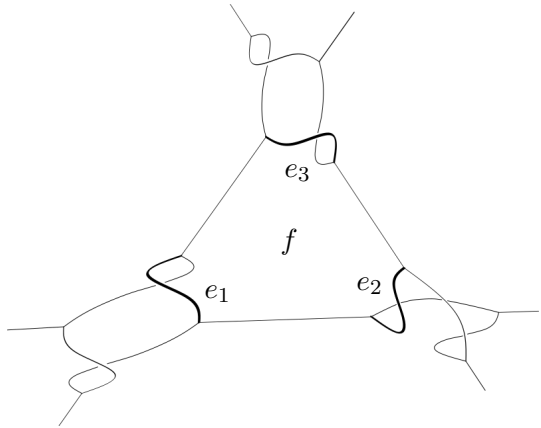
\begin{figure}[ht]
\centering
\begin{center}
\begin{picture}(70, 56)(0,0)
\put(0,0){\includegraphics[width = 70mm]{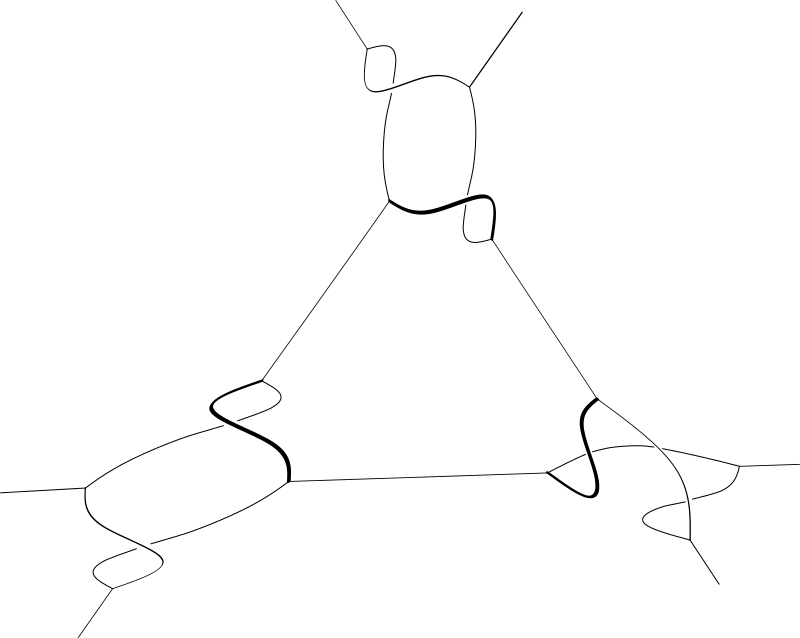}}

\put(26, 17){$e_1$}
\put(46, 18){$e_2$}
\put(36.5, 33){$e_3$}
\put(36, 23.5){$f$}
\end{picture}
\end{center}
\caption{Edges used for the slide correction.}
\label{fig:correctioncorner}
\end{figure}

\begin{definition}
Let~$r:\Sigma_1\to\Sigma_2$ be a partial return map around~$f$. We define~$\cor:\Sigma_2\to\Sigma_1$ the composition of slides along every~$e_i$ for~$1\leq i\leq k$. We call it the \textbf{slide correction} of~$r$.
\end{definition}

The diffeomorphism~$\cor$ is well-defined. Indeed~$f$ is a sink face so the slides are well-defined, and the slides on different corners can be done independently in a commutative way. The diffeomorphism~$\cor\circ r$ will be compared to the Dehn twist along~$\gamma_f$, for the curve~$\gamma_f$ represented in Figure \ref{fig:curves}. This curve does one turn around~$f$, and follows the edge~$e_i$ for each corner~$c_i$ of~$f$.

\begin{proposition}\label{Dehntwistface}
Let~$\eta$ and~$\nu$ be two Eulerian coorientations that differ only by an elementary flip along a sink face~$f$. Let~$r:\Sigma_\eta\to\Sigma_\nu$ be the partial return map and~$\cor_r:\Sigma_\nu\to\Sigma_\eta$ the corresponding slide correction. Then~$\cor_r\circ r$ is isotopic to the negative Dehn twist along~$\gamma_f$.
\end{proposition}

\begin{proof}
We start with an additional assumption on~$f$: we suppose that~$f$ does not admit double corners as an immersed polygon. That is, we suppose that~$f$ is an embedded polygon. First we see that there is an annulus containing the support of~$\cor_r\circ r$. Denote by~$U$ the union of the complement of a small neighbourhood of~$f$ and of the opposite sides of~$e_i$ for every corner~$c_i$ of~$f$ (the opposite side in the ribbon graph of~$\Sigma_\eta$ around a corner~$c_i$). Denote~$V=\Sigma_1\setminus U$. We can do this choice so that~$V$ is homeomorphic to an annulus that retracts on~$\gamma_f$, as in Figure \ref{partialreturn3}.

The remark comes from the fact that both~$\cor_r$ and~$r$ do not act on~$U$. We mean by "not act" that the immersions~$\pi\circ \psi_{im}$ and~$\pi\circ \psi_{im}\circ \cor_r\circ r$ are equal on~$U$. Thus they lift to the equality~$(\cor_r\circ r)_{|U}=\Id_{U}$ and~$\supp(\cor_r\circ r)\subset V$. So~$\cor_r\circ r$ is isotopic to a multiple of the negative Dehn twist along~$\gamma_f$.

In order to understand which multiple it is, we use an arc~$\delta$ transversal to~$V$, and we compare~$\delta$ to~$\cor_r\circ r\circ\delta$. Let~$x,y\in\dd f$ be two points that are not corners of~$f$. Suppose that the smoothings of~$\Sigma_1$ and~$\Sigma_2$ have been done away from~$x$ and~$y$, so that~$\gamma_x=T^1_xS\cap\Sigma_i$ and~$\gamma_y=T^1_y S\cap\Sigma_i$ are two arcs that do not depend on~$i$. Once~$\delta$ will be defined, since~$\gamma_y$ intersects the core of~$V$ only once, the multiplicity of the Dehn twist is equal to the algebraic intersection of~$[r(\delta)-\delta].[\gamma_y]$.

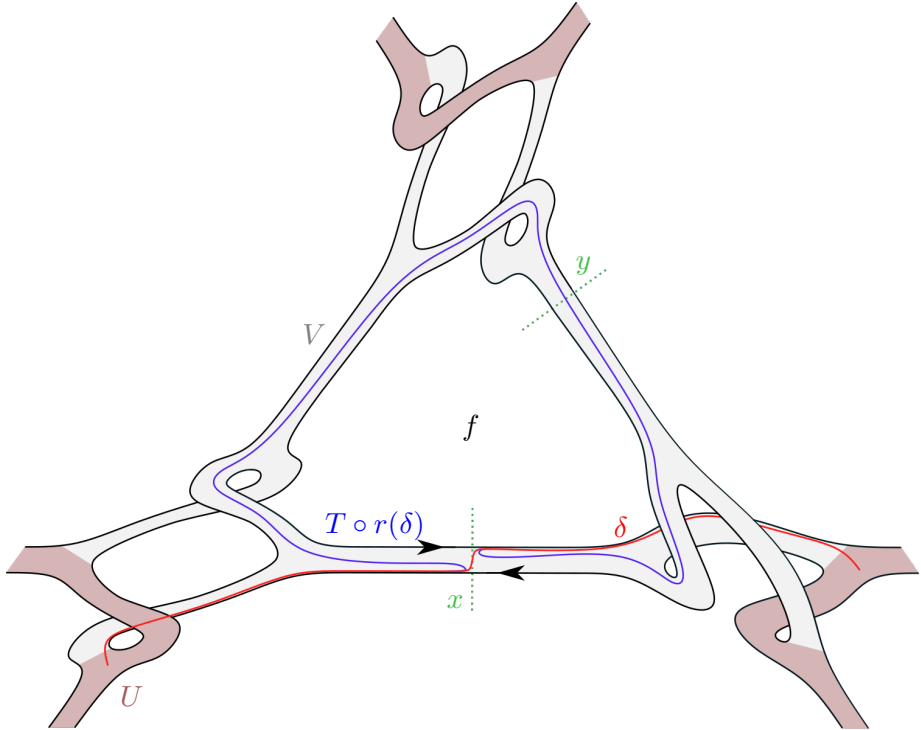
\begin{figure}[ht]
\centering
\begin{center}
\begin{picture}(120, 96)(0,0)
\put(0,0){\includegraphics[width = 120mm]{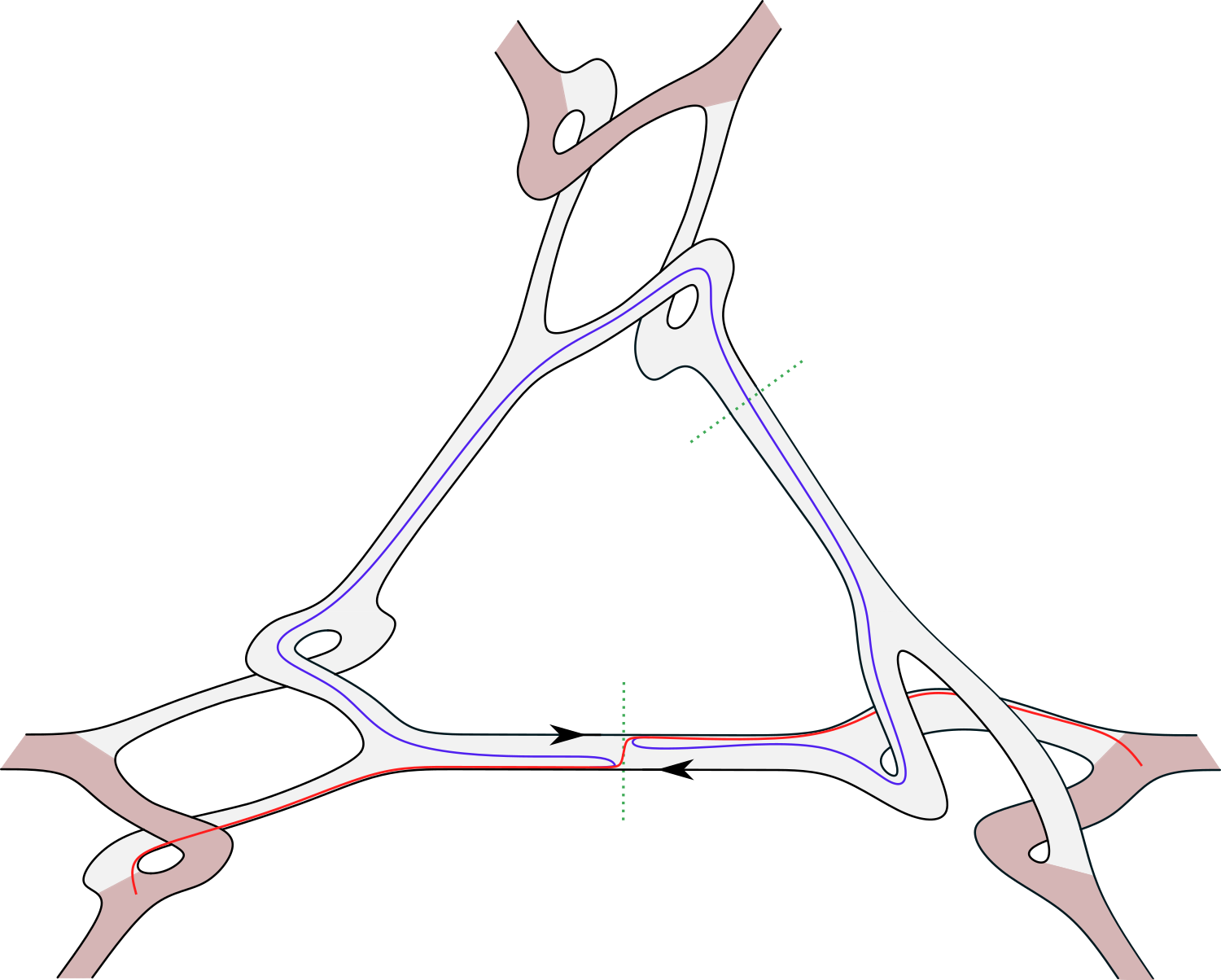}}

\put(60, 39){$f$}
\color{red}
\put(80, 25.5){$\delta$}
\color{blue}
\put(42, 26){$T\circ r(\delta)$}
\color{gray!50!green}
\put(58, 16){$x$}
\put(75, 61){$y$}
\color{gray}
\put(39, 51){$V$}
\color{gray!70!red}
\put(15, 3){$U$}
\color{black}

\end{picture}
\end{center}
\caption{Action of the partial return map around an alternating vertex.}
\label{partialreturn3}
\end{figure}

Take two arcs~$\gamma_x^+$ and~$\gamma_x^-$ in~$\dd\Sigma_1$, that start from the ends of~$\gamma_x$ and end in~$U$. Then define~$\delta$ to be an arbitrary closed smoothing of~$\delta_x^-\cup\delta_x\cup\delta_x^+$ that remains in~$\Sigma_1$, as in Figure~\ref{partialreturn3}. For~$p\in\delta$ arbitrary close to~$\dd\Sigma_1$,~$r_f(p)$ is not in~$\gamma_y$ since~$x$ and~$y$ are on different faces of~$f$. Thus~$r_f(\delta)$ intersects~$\gamma_y$ only once, corresponding to the geodesic in~$f$ between~$x$ and~$y$. Also by construction,~$\delta\cap\gamma_y=\emptyset$. Since~$\cor_r$ restrict to the identity outside a small neighbourhood of~$\Gamma_0\cap S$, we have~$[\cor_r\circ r(\delta)-\delta].[\gamma_y]=\pm 1$, the multiplicity is~$\pm 1$. Figure \ref{partialreturn3} shows in blue~$\cor_r\circ r(\delta)$, which helps finding the sign. To know more precisely the sign, one could detail how the orientation of~$\partial \Sigma_\eta$ imposes to~$\cor_r\circ r(\delta)$ to intersect~$\Sigma\cap T^1 y$ with this sign. So~$\cor_r\circ r$ is isotopic to a negative Dehn twist along~$\gamma_f$.

To prove the property in the general case, we could either use a covering of~$S$ such that~$f$ lifts to an embedded polyhedron, or adapt the last argument around the double corners (and see why~$V$ is still an annulus).
\end{proof}

\subsection{Reconstruction of the first-return map}

We now understand the first-return map as product of simple maps. We first define the curves appearing in Theorem \ref{introtheoremCrestated}. Then we restate and prove Theorem \ref{introtheoremCrestated}.

\begin{figure}
\centering
\includegraphics[scale=0.24]{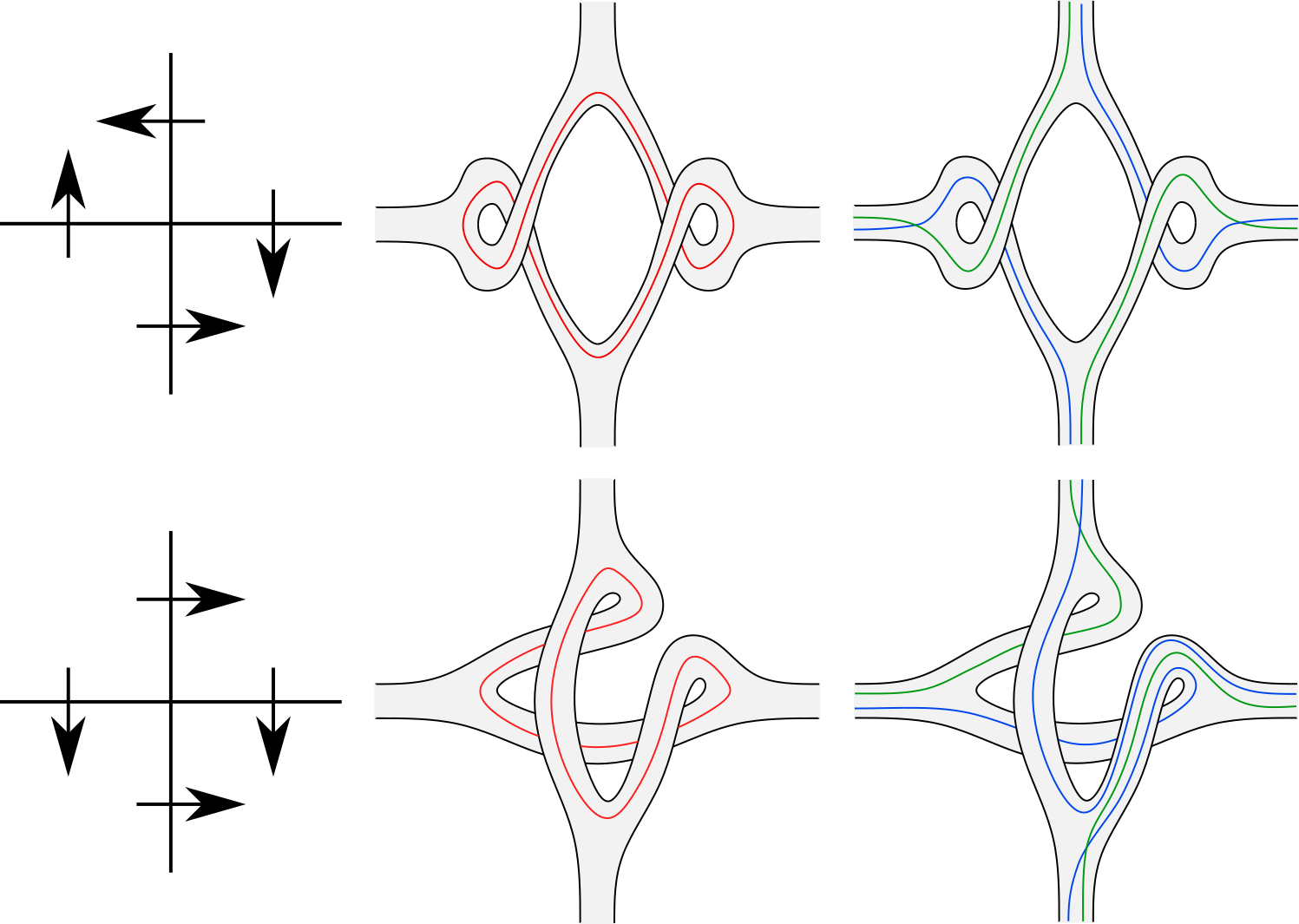}
\caption{Curves~$\gamma_v$ for a vertex~$v$ (in the middle), and~$\gamma_f$ for a face~$f$ (in the right).}
\label{fig:gammavf}
\end{figure}

\begin{definition}\label{def:curves}
For every vertex~$v\in\Gamma_0$, define the curve~$\gamma_v$ as the skeleton of the annulus~$\Sigma\cap T^1_BS$ for~$B\subset S$ a small ball around~$v$, as in Figure \ref{fig:gammavf}.

For each face~$f\in\Gamma_2$, define the curve~$\gamma_f$ in~$\Sigma_\eta$ that does one turn around~$f$, such that the behavior of~$\gamma_f$ around a corner of~$f$ is as in Figure \ref{fig:gammavf}.

If needed, we denote by~$\gamma_x^\eta$ the curve along~$x\in\Gamma_0\cup\Gamma_2$ for the coorientation~$\eta$.
\end{definition}

An example of a full~$\gamma_f$ is presented in Figure \ref{fig:curves}.

\setcounter{introtheorem}{2}
\begin{introtheorem}\label{introtheoremCrestated}
Let~$\eta$ be an acyclic Eulerian coorientation and~$\Sigma_\eta$ its corresponding Birkhoff section. Then the first-return map~$r:\Sigma_\eta\to\Sigma_\eta$ is the product of negative Dehn twists along~$\gamma_v$ for all~$v\in\Gamma_0$ and~$\gamma_f$ for all~$f\in\Gamma_0^\star$. The product is ordered by any representation of~$\eta$ (write from left to right from highest to lowest).
\end{introtheorem}

\begin{remark}
According to Remark \ref{commutation}, we could have taken other curves and make them appear in a different order. We took a convention that depends mainly on the choice of the Eulerian coorientation.
\end{remark}

\begin{proof}[Proof of Theorem \ref{introtheoremCrestated}]
For this proof, we denote by~$\Tau\gamma$ the negative Dehn twist along~$\gamma$, for a simple curve~$\gamma$.

Take a representation~$\preceq$ of~$\eta$ and order the faces of~$\Gamma$ by~$f_1\preceq\hdots\preceq f_n$. Denote by~$\eta_0=\eta$, successively~$\eta_i=I_{f_i}(\eta_{i-1})$ the flip of~$\eta_{i-1}$ along~$f_i$, so that~$\eta_n=\eta$, and denote by~$\Sigma_i=\Sigma_{\eta_i}$. According to Theorem \ref{introtheoremB}, the first-return map is a product of partial return maps~$r_n\circ\hdots\circ r_1$ for the partial return map~$r_i:\Sigma_{i-1}\to\Sigma_i$ along the face~$f_i$. Denote by~$\cor_i:\Sigma_{i}\to\Sigma_{i-1}$ the slide correction of~$r_i$, and for~$1\leq i\leq n$ define~$g_i=\cor_1\circ\hdots\circ \cor_i$. According to Proposition \ref{Dehntwistface},~$\cor_i\circ r_i$ is isotopic to a negative Dehn twist along~$\gamma_{f_i}^{\eta_{i-1}}$, so that:

\begin{align*}
r &= r_n\circ\hdots\circ r_1 \\
 &= \prod_{i=n}^1 r_i \\
 &= \prod_{i=n}^1 g_i^{-1}\circ g_{i-1}\circ \cor_i\circ r_i\\
 &= g_n^{-1}\circ \prod_{i=n}^1 g_{i-1}\circ \cor_i\circ r_i\circ g_{i-1}^{-1} \\
r &= g_n^{-1}\circ \prod_{i=n}^1 g_{i-1}\circ \Tau\gamma_{f_i}^{\eta_{i-1}}\circ g_{i-1}^{-1}
\end{align*}

We will see that~$g_n^{-1}:\Sigma_\eta\to\Sigma_\eta$ is a commutative product of Dehn twists, and then we characterize the curve along which~$g_{i-1}\circ \Tau\gamma_{f_i}^{\eta_{i-1}}\circ g_{i-1}^{-1}$ is a Dehn twist. We will also change the order in which the Dehn twists appear.

We claim that:~$$g_n^{-1}=\prod_{v\in\Gamma_0}\Tau{\gamma_v^{\eta}}$$

The isotopy~$g_n^{-1}$ is a concatenation of slides (which are isotopies). Since all faces appear in the product~$g_n=\cor_1\circ\hdots\circ \cor_n$, the slide for every corner of~$\Gamma$ appears in~$g_n^{-1}$. If two slides appear on different vertices of~$\Gamma$, they have disjoint supports. So we can rearrange the slides so that they appear by groups of four, one group for every vertex of~$\Gamma$. The way we have constructed~$g_n^{-1}$ ensures that the slides in each group appear in the same order than in Lemma \ref{lemma:DehnTwistVertex}. So the lemma proves that~$g_n^{-1}$ is a product of negative Dehn twists along the curve~$\gamma_v$.

We proved that~$r$ is the product of negative Dehn twists given by the following product, where~$g_{i-1}\circ \Tau\gamma_{f_i}^{\eta_{i-1}}\circ g_{i-1}^{-1}$ is a pullback of a negative Dehn twist into~$\Sigma_\eta$:
\[r = \prod_{v\in\Gamma_0}\Tau{\gamma_v^{\eta}}\circ \prod_{i=n}^1 (g_{i-1}\circ \Tau\gamma_{f_i}^{\eta_{i-1}}\circ g_{i-1}^{-1})\]

 Also~$g_{i-1}\circ \Tau\gamma_{f_i}^{\eta_{i-1}}\circ g_{i-1}^{-1}=\Tau(g_{i-1}^{-1}(\gamma_f^{\eta_{i-1}}))$ (see Remark \ref{commutation}). Here,~$f$ is the~$i$-th face given by the representation, and~$\gamma_f$ is the curve of~$\Sigma_i$ along~$f$.

\begin{remark}\label{commutation}
Recall that for~$g$ a orientation-preserving diffeomorphism of~$S$ and~$\gamma$ a simple closed curve on~$S$, we have~$g\circ \Tau\gamma\circ g^{-1}=\Tau{g(\gamma)}$ (Fact 3.7 of \cite{farb2012primer}). So if~$\gamma$ and~$\delta$ are two simple closed curves on~$S$, we have 
\begin{align*}
	\Tau\delta\circ\Tau\gamma &= \Tau(\Tau\delta(\gamma))\circ\Tau\delta\\
		&= \Tau\gamma\circ\Tau(\Tau^{-1}\gamma(\delta))
	\end{align*}
\end{remark}

Instead of computing~$g_{i-1}^{-1}(\gamma_f^{\eta_{i-1}})$, we first change the order of Dehn twists in the product, so that they appear in the order prescribed by the representation. In order not to increase the number of notations and indices, we will do an informal proof. We already know that~$r$ is a product of curves corresponding to all vertices and all faces of~$\Gamma$. Also the Dehn twists corresponding to the faces are already in positions prescribed by the induced partial representation. We first change the order in our product, then prove that the Dehn twists are along the announced curves.

A curve~$\gamma_v^\eta$ can only intersect~$g_{i-1}^{-1}(\gamma_{f_i}^{\eta_{i-1}})$ if~$f_i$ admits~$v$ as corner. So we can change the position of~$\gamma_v^\eta$ in the product until~$\Tau{\gamma_v^\eta}\circ \Tau(g_{i-1}^{-1}(\gamma_{f_i}^{\eta_{i-1}}))$ appears in the product. According to Remark \ref{commutation}, we have:

\begin{equation}
\label{equation}
\Tau{\gamma_v^\eta}\circ \Tau(g_{i-1}^{-1}(\gamma_{f_i}^{\eta_{i-1}}))=\Tau(\Tau\gamma_v^\eta(g_{i-1}^{-1}(\gamma_{f_i}^{\eta_{i-1}})))\circ \Tau{\gamma_v^\eta}
\end{equation}

We use this equation for changing the position of the two curves in the product. We repeat the procedure until~$\Tau{\gamma_v}$ is at the place prescribed by the representation. We do the same procedure for all~$v\in\Gamma_0$.

It remains to see along which curves is the Dehn twist associated to a face~$f$. We push~$\gamma_{f_i}^{\eta_{i-1}}$ back in~$\Sigma_\eta$ with~$g_{i-1}$, which is a series of slides. Only the ones along corners of~$f$ matter. We then compose it with~$\Tau{\gamma_v}$ for the corners~$v$ of~$f$ that are ordered before~$f$, which is also a series of slides. We proved that~$g_{n}^{-1}$ is the product of~$\Tau{\gamma_v}$, so by the slides induced by the~$\Tau{\gamma_v}$ are reversed slides of some of the slides induced by~$g_{i-1}$. It remains to take a corner~$v$ of~$f$, and see as in Figure \ref{svertexisotopy1}, how the slides act on~$\gamma_{f_i}^{\eta_{i-1}}$, depending on the representation around~$v$.

We do one case as an example, the other cases are similar. In Figure \ref{fig:determination} is presented the case of a non-alternating vertex~$v$. Suppose there is no double corner and the face~$f$ is in third position around~$v$. We represent~$\gamma_{f_3}^{\eta_2}\in\Sigma_2$, and we push it into~$\Sigma_0$ using slides along the first and second corner. Since~$v<f$ for the representation, we use Equation~$\eqref{equation}$ to place~$v$ and~$f$ in relative position in the product. The corresponding curve, represented in the right, is isotopic to~$\gamma_f$ from Definition \ref{def:curves}.

\begin{figure}
\centering

\begin{center}
\begin{picture}(140, 33)(0,0)
\put(0,0){\includegraphics[width = 140mm]{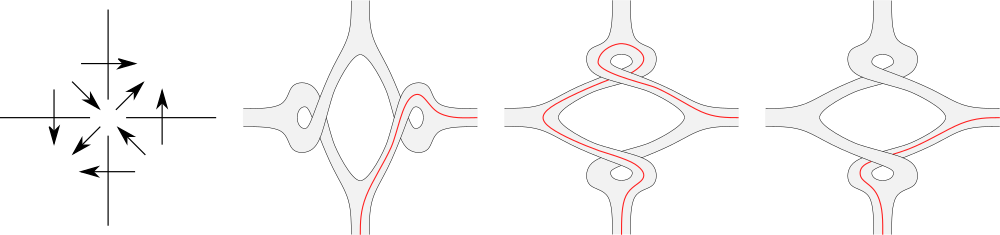}}

\put(14, 15.5){$v$}
\put(21.5, 23){$1$}
\put(6.5, 8){$2$}
\put(21.5, 8){$3$}
\put(6.5, 23){$4$}

\put(60, 3){$\Sigma_2$}
\put(96, 3){$\Sigma_0$}
\put(133, 3){$\Sigma_0$}

\color{red}
\put(48, 16){$\gamma_{f_3}^{\eta_2}$}
\put(80, 16){$g_2^{-1}(\gamma_{f_3}^{\eta_2})$}
\put(113, 16){$T\gamma_v g_2^{-1}(\gamma_{f_3}^{\eta_2})$}
\color{black}

\end{picture}
\end{center}
\caption{Determination of~$\gamma_f$ around a non-alternating vertex in third position.}
\label{fig:determination}
\end{figure}
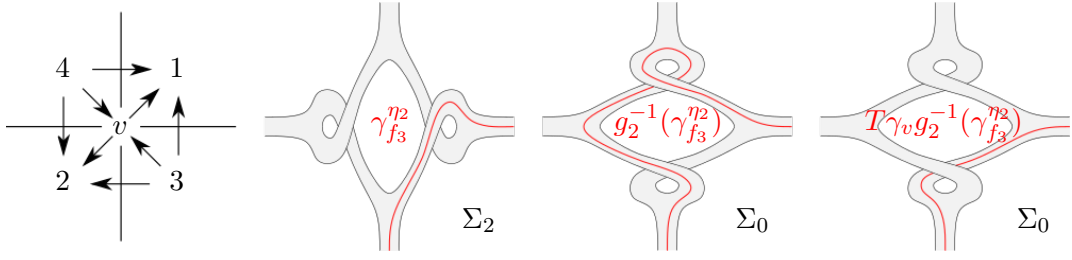

\end{proof}

\subsection{Comparison of different Eulerian coorientations}

In this subsection, we compare the explicit products of negative Dehn twists for different representations or different acyclic Eulerian coorientations. We will in particular prove Corollary~\ref{introcorollaryD}.

Assume that~$\eta$ is an acyclic Eulerian coorientation, so that the surface~$\Sigma_\eta$ is a Birkhoff section. Given two representations, the curves~$\gamma_v$ and~$\gamma_f$ depend on~$\eta$ only. 

\begin{lemma}\label{lemma:comparationsamecoorientation}
Let~$\preceq_1$ and~$\preceq_2$ be two representations of~$\eta$. The two products of negative Dehn twists in Theorem \ref{introtheoremCrestated} can be changed one into another one by successively swapping the positions of consecutive commutating Dehn twists.
\end{lemma}

\begin{proof}
Denote by~$\preceq_0$ the coherent ordering, which is a partial ordering on~$\Gamma_0\cup\Gamma_2$. By definition,~$\preceq_1$ and~$\preceq_2$ agree with~$\preceq_0$. Let~$x,y$ in~$\Gamma_0\cup\Gamma_2$ and suppose that~$\gamma_x$ and~$\gamma_y$ do not commute. Then by definition of the curves~$\gamma_x$ and~$\gamma_y$,~$x$ and~$y$ must be either two adjacent faces or one face and an adjacent vertex, thus they are comparable under~$\preceq_0$. So~$x$ and~$y$ have the same ordering under~$\preceq_1$ and~$\preceq_2$. 

Now suppose that~$\preceq_1$ and~$\preceq_2$ are not equal, and let~$x$ and~$y$ be in~$\Gamma_0\cup\Gamma_2$ not ordered in the same way by~$\preceq_1$ and~$\preceq_2$. Such elements~$(x,y)$ can be taken consecutive in~$\preceq_2$, otherwise~$\preceq_1$ and~$\preceq_2$ would be equal. By what precedes,~$x$ and~$y$ are not adjacent, and the negative Dehn twists along~$\gamma_x$ and~$\gamma_y$ commute. So we can define~$\preceq_3$ by only swapping in~$\preceq_2$ the ordering of~$x$ and~$y$, and~$\preceq_3$ is a representation of~$\eta$. This procedure can be recursively repeated to~$\preceq_1$ and~$\preceq_3$, and terminates in a finite number of steps.
\end{proof}

Given two isotopic Birkhoff sections and decompositions of the two first return maps in Dehn twists, we can use the isotopy to compare the decompositions. We denote by~$\gamma_x^\nu$ the curve associated to~$x\in\Gamma_0\cup\Gamma_2$ and for the coorientation~$\nu$.

\begin{lemma}
Let~$\eta$ and~$\nu$ be two cohomologous acyclic Eulerian coorientations. Then the two decompositions of the first return map given by Theorem \ref{introtheoremCrestated} for~$\eta$ and~$\nu$ are Hurwitz equivalent. 
\end{lemma}
	
\begin{proof}
We first consider the case for~$\eta$ an acyclic coorientation,~$f$ a sink face of~$\eta$ and~$\nu=I_f(\eta)$ the cohomologous coorientation obtained from~$\eta$ by flipping~$f$. The partial return map~$r_f:\Sigma_\eta\to\Sigma_\nu$ is a diffeomorphism, that we can use to compare the Dehn twists on both surfaces. Let~$f'$ be a face of~$\Gamma$. If~$f'$ and~$f$ are not adjacent, then~$r_f(\gamma_{f'}^\eta)=\gamma_{f'}^\nu$, and the Dehn twists~$\gamma_{f}^\nu$ and~$\gamma_{f'}^\nu$ commute. If~$f'$ and~$f$ are adjacent,~$r_f(\gamma_{f'}^\eta)$ and~$\gamma_{f'}^\nu$ differ by a Dehn twist along~$\gamma_f^\nu$.  

There exist two representations~$\preceq_\eta$ and~$\preceq_\nu$ of~$\eta$ and~$\nu$ that differ only for~$f$, so that~$f$ is the minimum of~$\preceq_\eta$ and the maximum of~$\preceq_\nu$. Consider the decompositions in Dehn twists given by Theorem \ref{introtheoremCrestated} for these representations. There is a Hurwitz equivalence between both decompositions, that consists in successively swapping the position of~$\tau_{\gamma_f}$ with its neighbor Dehn twists and conjugating the neighbor Dehn twists by~$\tau_{\gamma_f}$. 

Now consider~$\eta$ and~$\nu$ any two cohomologous acyclic Eulerian coorientations. According to Lemma \ref{lemma:comparationsamecoorientation}, the Hurwitz equivalence class of the decomposition does not depend on the choice of a representation. Thanks to Proposition \ref{proposition:coorientationEquivalence} in Appendix \ref{appendix: comparison}, there exists a finite sequence of flips that change~$\eta$ into~$\nu$. One can use this sequence and the previous paragraph to compare the two decompositions of the first return map produced by Theorem \ref{introtheoremCrestated}.
\end{proof}

Corollary \ref{introcorollaryD}, proved below, proposes an alternative comparison, even for cohomologous coorientations.

\begin{introcorollary}\label{introcorollaryD}
Let~$S$ be a hyperbolic surface,~$\Gamma$ a finit collection of closed geodesics on~$S$, and consider the geodesic flow on~$T^1S$. There exists a common combinatorial model~$\Sigma_\Gamma$ for all Birkhoff sections with boundary~$-\darrow{\Gamma}$, and an explicit family of simple closed curves~$\gamma_1, \cdots, \gamma_n$ in~$\Sigma_\Gamma$ such that the first-return maps for these Birkhoff sections are product of negative Dehn twists of the form~$\Tau{\gamma_{\sigma(1)}}\circ\dots\circ\Tau{\gamma_{\sigma(n)}}$ for some permutation~$\sigma$ of~$\{1, \cdots, n\}$. 
\end{introcorollary}

\begin{proof}
Let~$\eta$ be any acyclic Eulerian coorientation of~$\Gamma$. Define~$\Sigma_\Gamma=\Sigma_\eta$, and~$\{\gamma_i\}=\{\gamma_f,f\in\Gamma_2\}\cup\{\gamma_v,v\in\Gamma_0\}$. Every Birkhoff section in the fibered face is isotope to a Birkhoff section~$\Sigma_\nu$ for~$\nu\in\coorientation{\Gamma}$ acyclic. We will do a finite sequence of slides to compare~$\Sigma_\nu$ to~$\Sigma_\Gamma$, together with the curves they hold.

Let~$v$ be a vertex in~$\Gamma_0$. Up to symmetry and rotation, there are nine configurations for~$(\eta,\nu)$ around~$v$. For each configuration, we can do one or two slides to isotope~$\Sigma_\eta$ to~$\Sigma_\nu$ around~$v$. Denote~$sl:\Sigma_\eta\to\Sigma_\nu$ the diffeomorphism induced by the sequence of slides. We can compare~$\gamma_f^\nu$ and~$sl(\gamma_f^\eta)$. In each case~$\gamma_v^\nu = sl(\gamma_v^\eta)$, and for any face~$f$ adjacent to~$v$, either~$\gamma_f^\nu$ and~$sl(\gamma_f^\eta)$ are equal, or they differ by a Dehn twist along the curve~$\gamma_v$ (positive of negative). And in every case, we can apply Theorem \ref{introtheoremCrestated} to~$\nu$, and obtain a product of negative Dehn twists along the curves~$\gamma_v^\nu$ and~$\gamma_f^\nu$. We can swap positions of consecutive Dehn twists, including~$\gamma_v^\nu$ and~$\gamma_f^\nu$, to obtain a product of Dehn twists along the curves~$sl(\gamma_v^\eta)$ and~$sl(\gamma_f^\eta)$. We will detail one case, the others being similar.

\begin{figure}
	\centering
	
	\begin{center}
	\begin{picture}(70, 60)(0,0)
	\put(0,0){\includegraphics[width = 70mm]{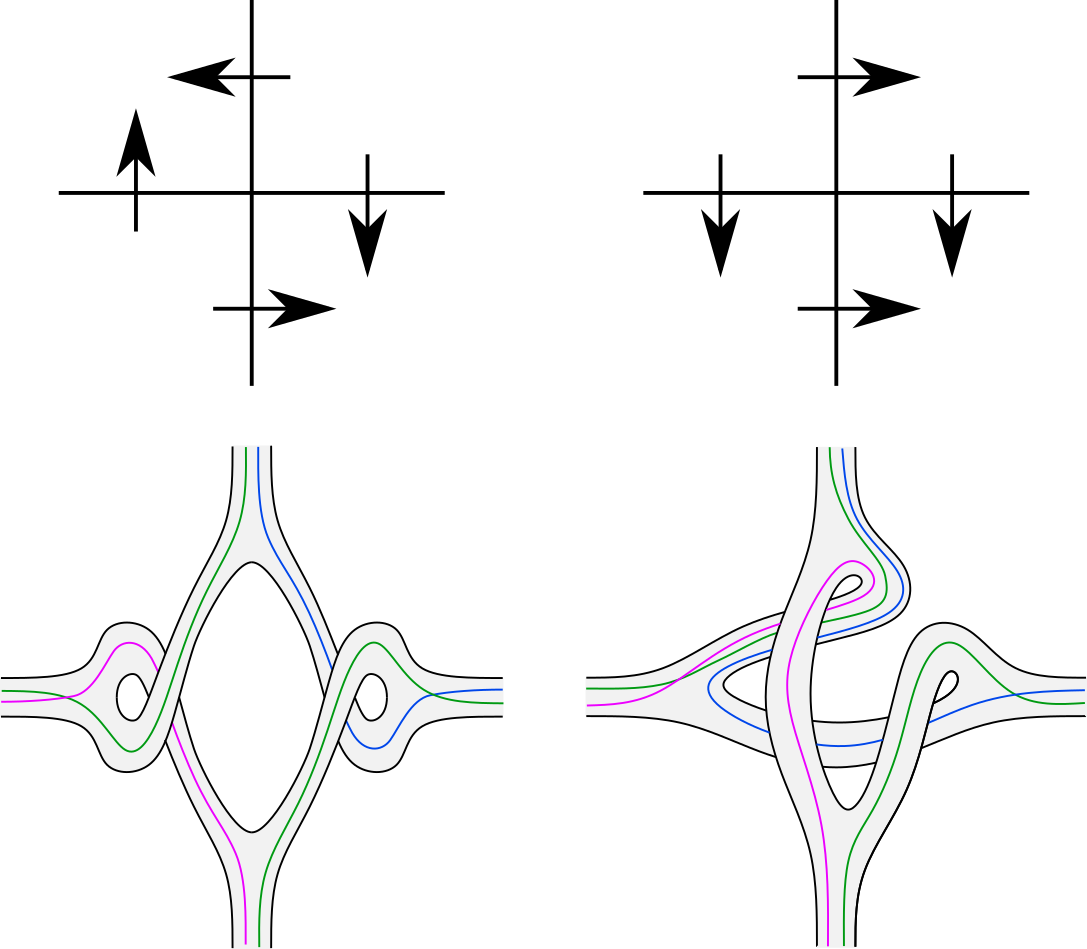}}
	
	\put(7.7, 55){$1$}
	\put(22.5, 55){$2$}
	\put(7.7, 40){$3$}
	\put(22.5, 40){$4$}
	
	\end{picture}
	\end{center}
	\caption{Action of a slide on the curves~$\gamma_f$. The slide is done in the quadrant~$1$. }
	\label{fig:alternativegammavf}
\end{figure}
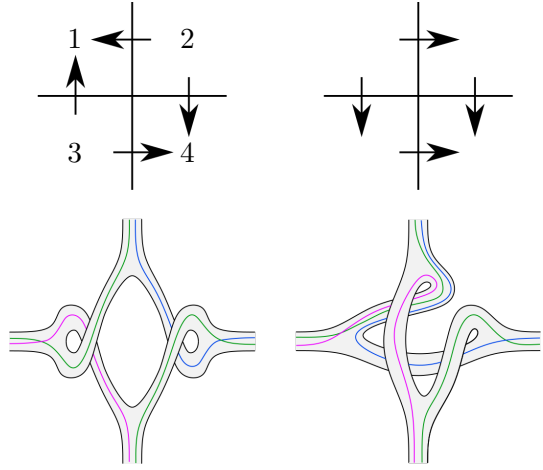

Consider the coorientation~$\eta$ (left) and~$\nu$ (right) presented in Figure \ref{fig:alternativegammavf}. On the figure, we represent~$\gamma_f^\eta$ on the left and~$sl(\gamma_f^\eta)$ on the right, for the four faces~$f$ adjacent to~$v$. We have~$sl(\gamma_{f_i}^\eta)=\gamma_{f_i}^\nu$ for~$i\in\{1,4\}$, and~$sl(\gamma_{f_j}^\eta)=\tau_{\gamma_v}(\gamma_{f_j}^\nu)$ for~$j\in\{2,3\}$, where~$\tau_{\gamma_v}$ is the negative Dehn twist along~$\gamma_v$. Let~$\preceq$ be a representation of~$\nu$, so that up to changing~$2$ and~$3$, we have~$f_4\preceq f_3\preceq f_2\preceq v\preceq f_1$. The first return map of~$\Sigma_\nu$ given by Theorem \ref{introtheoremCrestated} contains a sub-product of the form~$T\gamma_{f_1}^\nu\circ\cdots\circ T\gamma_v^\nu\circ \cdots\circ T\gamma_{f_2}^\nu\circ \cdots\circ T\gamma_{f_3}^\nu\circ \cdots\circ T\gamma_{f_4}^\nu$. But~$T\gamma_v^\nu$ commutes with any Dehn twist that is not a~$T\gamma_{f_i}$ for~$1\leq i\leq 4$. According to remark \ref{commutation}, for~$i=2,3$ we have~$\gamma_v^\nu\circ T\gamma_{f_i}^\nu=T(T\gamma_v(\gamma_{f_i}^\nu)\circ T\gamma_v^\nu=Tsl(\gamma_{f_i}^\eta)\circ \Tau\gamma_v^\nu$.

So together with Remark \ref{commutation}, we can change the position of~$T\gamma_v^nu$ so that the Dehn twists appear in the order~$f_1,f_2,f_3,v,f_4$, and are along the curve~$sl(\gamma_{f_i}^\eta)$ and~$sl(\gamma_v^\eta)$. We can do this procedure for all vertices~$v\in\Gamma_0$, which prove that there exists a diffeomorphism~$sl:\Sigma_\Gamma\to\Sigma_\nu$ so that the first return map on~$\Sigma_\nu$ is a product of negative Dehn twist along the curve~$sl(\gamma_f^\eta)$ and~$sl(\gamma_v^\eta)$, whose ordering depends on~$\nu$.

\end{proof}

\bibliographystyle{alpha}
\bibliography{bibliography}

\begin{thebibliography}{McM00}

\bibitem[A'C98]{A'Campo98}
Norbert A'Campo.
\newblock Generic immersions of curves, knots, monodromy and gordian number.
\newblock {\em Publications Math\'ematiques de l'IH\'ES}, 88:151--169, 1998.

\bibitem[Bir17]{Birkhoff314}
George~D. Birkhoff.
\newblock Dynamical systems with two degrees of freedom.
\newblock {\em Trans. Amer. Math. Soc. 18 (1917), 199-300}, 1917.

\bibitem[CD16]{CossariniDehornoy}
Marcos Cossarini and Pierre Dehornoy.
\newblock Intersection norms on surfaces and birkhoff cross sections, 2016.

\bibitem[Deh15]{Dehornoy2015}
Pierre Dehornoy.
\newblock Geodesic flow, left-handedness and templates.
\newblock {\em Algebr. Geom. Topol.}, 15(3):1525--1597, 2015.

\bibitem[DL19]{dehornoy2019divide}
Pierre Dehornoy and Livio Liechti.
\newblock Divide monodromies and antitwists on surfaces, 2019.

\bibitem[FM12]{farb2012primer}
B.~Farb and D.~Margalit.
\newblock {\em A Primer on Mapping Class Groups}.
\newblock Princeton Mathematical Series. Princeton University Press, 2012.

\bibitem[Fri82a]{Fried1982}
David Fried.
\newblock Flow equivalence, hyperbolic systems and a new zeta function for
  flows.
\newblock {\em Commentarii Mathematici Helvetici}, 57(1):237--259, 1982.

\bibitem[Fri82b]{Fried82section}
David Fried.
\newblock The geometry of cross sections to flows.
\newblock {\em Topology}, 21(4):353 -- 371, 1982.

\bibitem[Ish04]{Ishikawa2004}
Masaharu Ishikawa.
\newblock Tangent circle bundles admit positive open book decompositions along
  arbitrary links.
\newblock {\em Topology}, 43(1):215 -- 232, 2004.

\bibitem[McM00]{McMullen00}
Curtis~T McMullen.
\newblock Polynomial invariants for fibered 3-manifolds and teichmüller
  geodesics for foliations.
\newblock {\em Annales Scientifiques de l’École Normale Supérieure},
  33(4):519 -- 560, 2000.

\bibitem[Pro02]{propp2002}
James Propp.
\newblock Lattice structure for orientations of graphs, 2002.

\bibitem[Sch57]{Schwartzman1957}
Sol Schwartzman.
\newblock Asymptotic cycles.
\newblock {\em Annals of Mathematics}, 66(2):270--284, 1957.

\end{thebibliography}

\begin{appendices}

\section{Construction of an explicit coorientation with fixed cohomology \label{appendix: construction}}

Fix a collection~$\Gamma$ of geodesic curve in~$S$. An Eulerian coorientation~$\eta$ of~$\Gamma$ induces an element~$[\eta]$ in~$H^1(S,\Z)$, that counts the algebraic intersection of a curve~$\gamma$ with~$(\Gamma, \eta)$. Note that the parity of~$[\eta]$ is fixed, since~$[\eta]\equiv[\Gamma] \mod 2$. We fix~$\omega\in H^1(S,\Z)$ with the good parity, and try to construct, if possible, an Eulerian coorientation with this cohomology. The ideas are already present in \cite{CossariniDehornoy}. We express them in an algorithmic manner, using the principle of Dijkstra's algorithm.

We denote by~$(\hat S, \hat\Gamma)$ the universal covering of~$(S, \Gamma)$. In order to define Eulerian coorientations, we use height functions. A height function is a function~$h:\{\text{faces of~$(\hat{S},\Gamma)$}\}\to\Z$ so that for any two adjacent faces~$f_1, f_2$,~$|h(f_1)-h(f_2)|=1$. We say that a height function~$h$ is~$\omega$ stable if for any loop~$\gamma$ of~$S$ and any lift~$\hat\gamma$ in~$\hat S$, we have~$h(\hat\gamma(1)) - h(\hat\gamma(0)) = \omega(\gamma)$.

\begin{lemma}
  There is a~$1:1$ correspondence between Eulerian coorientations of~$\Gamma$ with cohomology~$\omega$, and height functions on the universal covering~$(\hat S, \hat\Gamma)$ of~$(S, \Gamma)$, that are~$\omega$-stable.
\end{lemma}

\begin{proof}
We see an Eulerian coorientation as the gradient of height function, and the proof is straight-forward.
\end{proof}

We will construct a height function recursively, in the same way Dijkstra algorithm constructs the distance function. We denote by~$\bar S$ a closed fundamental domain of~$\hat S\to S$, that is an embedding on the interior of the faces. Note that a~$\omega$-stable height function can be recovered if we know its value only on the fundamental domain~$\bar S$.

We fix~$f_0$ an arbitrary face of~$S$, and we construct the maximal height function on~$\bar S$ with a fixed value on~$f_0$ and fixed homology class. The idea is summarized here and detail in pseudo-code later. We set~$h(f_0) = 0$, and define recursively a height function by exploring~$\bar S$ face by face. At each step, we take one face~$f$ which already has a height and so that we need to update its neighbourhood. Then compute a new potential height for each neighbor face~$f'$ of~$f$, using the homology of~$\omega$ if needed. If the new potential height of~$f'$ is lower than the old one, we update its height, and store~$f'$ in the list of faces that we later need to update their neighbourhood.

The pseudo-code in Algorithm \ref{algorithm:height} assumes we have~$\omega \in H^1(S,\Z)$ and a procedure to construct curves in~$S$. The pseudo code computes recursively three functions:~$height$ is the potential height function of~$\bar{S}$,~$state$ keeps track of which faces need to update, and~$path(f)$ contains a path ending at the face~$f$ that helps detecting if~$\omega$ is the cohomology of a coorientation. If we know that~$\omega$ is the cohomology of a coorientation, we can skip lines using~$path$, that is lines~$5,8$ and~$20-26$.

\begin{algorithm}\label{algorithm:height}
\begin{algorithmic}[1]
\Procedure{construct height}{}
\For {face~$f\neq f_0$}
	\State~$height(f) \gets +\infty$,~$state(f) \gets waiting$,~$path(f) \gets \emptyset$
\EndFor	
\State~$height(f_0)=0$,~$state(f_0) \gets to_evaluate$,~$path(f_0) \gets (f_0)$
\While {there is a face~$f$ with~$state(f)=to_evaluate$} 
	\State~$state(f)\gets waiting$
	\For {$e$ edge of~$f$}
		\State~$f' \gets$ the other adjacent face to~$e$
		\State~$h \gets height(f) + 1$
		\If {$e\in\partial\bar S$}
			\State~$\gamma \gets~$ a closed curve in~$S$ so that~$|\gamma\cap\partial\bar S|=1$ and oriented as~$f\xrightarrow{e} f'$
			\State~$h \gets h + \omega(\gamma)$
		\EndIf
		\If {$h<height(f')$}
			\State~$height(f') \gets h$
			\State~$state(f') \gets to_evaluate$
			\If {$f'\notin path(f)$}
				\State~$path(f') \gets path(f)\cup(e, f')$
			\Else{}
				\State decompose~$path(f)$ as~$(before,f',after)$
				\State~$\gamma \gets$ a closed curve in~$S$ matching~$(f', after, e, f')$
				\If {$|\gamma\cap\Gamma|<\omega(\gamma)$} 
					\Return There exist no valid height function
				\EndIf
				\State~$path(f') \gets (after)\cup(e, f')$
			\EndIf
		\EndIf 
	\EndFor	
\EndWhile
\State \:Return~$height$
\EndProcedure
\end{algorithmic}
\caption{Construction of a~$\omega$-stable height function.}
\end{algorithm}

\begin{proposition}
Let~$\omega$ in~$H^1(S,\Z)$. Algorithm \ref{algorithm:height} applied on~$\omega$ terminates and detects if~$\omega$ is the cohomology class of some Eulerian coorientation. In this case, it returns an associated height function.
\end{proposition}

\begin{proof}
We just give the idea of the proof.

\begin{lemma} \label{lemma:coorientationExistence} \cite{CossariniDehornoy}
There exists~$(\gamma_i)_{1\leq i\leq n}$ a family of closed curves in~$S$ transversal to~$\Gamma$ such that the following are equivalent: 
\begin{itemize}
\item there exists an Eulerian coorientation with cohomology class~$\omega$, 
\item all closed curves~$\gamma$ in~$S$ transversal to~$\Gamma$ satisfy~$|\gamma\cap\Gamma|\geq \omega(\gamma)$.
\item for all~$1\leq i\leq n$, we have~$|\gamma_i\cap\Gamma|\geq \omega(\gamma)$.
\end{itemize}
\end{lemma}

Suppose that an eulerian coorientation exists in the cohomology class~$\omega$. For a fixed face~$f$, the sequence of~$(height_n(f))_n$, for the step~$n$, is decreasing, with value in~$\Z\cup\{+\infty\}$ and minored by~$height(f) \geq height(f_0) - |\delta_f|$, where~$\delta_f$ is the shortest path in~$\Gamma^\star$ from~$f_0$ to~$f$. Thus the algorithm terminates. Also~$height$ is a height function corresponding to~$\omega$, which can be seen by expressing what it means for the algorithm to terminates.

The line~$25$ of the pseudo code tests this lemma on some well-chosen closed curves. The curves constructed in~$path$ satisfy some minimality property, which is required to test Lemma~\ref{lemma:coorientationExistence} efficiently. Notice that~$path$ does not necessary constructs the curves~$\gamma_i$ in Lemma~\ref{lemma:coorientationExistence}.
\end{proof}

\section{Equivalence of cohomologous coorientations.\label{appendix: comparison}}

We discuss a way to transform an acyclic Eulerian coorientation into any of its cohomologous coorientations by elementary flips~$I_f$. The combinatorial of Eulerian coorientations together with the flip transformation have already been studied by J.Propp, with the dual point of view. The following proposition is a mainly a geometric reformulation of Theorem~$1$ from \cite{propp2002}.

\begin{proposition}\label{proposition:coorientationEquivalence}
Let~$\eta,\nu$ be in~$\coorientation{\Gamma}$ be two cohomologous acyclic coorientations, (so that~$\Sigma_\eta$ is a Birkhoff section). Then there exists a sequence of elementary flips that change~$\eta$ into~$\nu$.

Let~$\eta,\nu$ be in~$\coorientation{\Gamma}$ two cohomologous coorientations that are not acyclic (so that~$\Sigma_\eta$ is not a Birkhoff section). Suppose that the union of oriented cycles in~$(\Gamma^\star,\eta)$ is connected. Then there exists a sequence of elementary flips that change~$\eta$ into~$\nu$.

Let~$\eta$ be in~$\coorientation{\Gamma}$ not acyclic. Suppose that the union of oriented cycles in~$(\Gamma^\star,\eta)$ is not connected. Then there exists~$\nu\in\coorientation{\Gamma}$ cohomologous to~$\eta$, so that~$\Sigma_\eta$ and~$\Sigma_\nu$ are not isotopic through the flow. In particular no sequence of flips can change~$\eta$ into~$\nu$.

\end{proposition}

Notice that we are never allow to flip a face included in an oriented cycle, and oriented cycles remain oriented the same way after any sequence of flips.

\begin{proof}
We start for~$\eta$ acyclic. Define~$E=\{e\in\Gamma_1,\eta(e)\neq\nu(e)\}$ and notice that~$E$ is an embedded graph in~$S$ with degree~$2$ and~$4$ vertices. Also~$\eta$ and~$\nu$ induce on~$E$ two opposite Eulerian coorientations. 

The dual graph~$E^\star$ is an acyclic oriented graph. Indeed take~$c$ a cycle in~$E^\star$. The two coorientations are cohomologous, so~$[\eta](c)=[\nu](c)$. Hence~$2[E](c)=([\eta]-[\nu])(c)=0$ and~$c$ cannot be an oriented cycle in~$E^\star$. Also every edge of~$E$ bounded two different connected components in~$S\setminus E$, or elsewhere we would have a closed curve~$c$ intersecting~$E$ only once, and with~$\eta(c)\neq\nu(c)$.

Hence~$\eta$ restricted to~$E$ induces a partial ordering on~$E^\star$. We will use this partial order to "solve" the problem by beginning by the local minimal elements. Take~$F^\star\in E^\star$ a local minimal for~$\eta$. On the boundary~$\partial F$,~$\eta$ is going inward. Since~$\eta$ is acyclic, the faces~$f\in\Gamma^\star$ with~$f\subset F$ are ordered by~$\eta$, and we can apply the flip algorithm to sub-faces of~$F$, flipping once every sub-faces of~$F$. After this procedure, we obtain an acyclic Eulerian coorientation~$\eta'$ cohomologous to~$\eta$, that differs only on~$\partial F$. So the difference between~$\eta'$ and~$\nu$ bounds one less connected component. By applying this procedure at most a finite number of time, we describe a finite number of elementary flips that transform~$\eta$ into~$\nu$.

Suppose now that~$\eta$ is not acyclic. Denote by~$U$ the union of oriented cycles in~$(\Gamma^\star,\eta)$, and suppose that~$U$ is connected, as a subgraph of~$\Gamma^\star$. Since~$\eta$ and~$\nu$ are cohomologous, an oriented cycle for~$\eta$ is also oriented for~$\nu$. So~$U$ is also the union of oriented cycles in~$(\Gamma^\star,\nu)$. We also denote by~$U\subset S$ the union of faces it induces. In~$(\Gamma^\star,\eta)$, there is no oriented path outside~$U$, starting and ending in~$U$, elsewhere this path would be a subset of an oriented cycle (since~$U$ is connected) and thus in~$U$. So every oriented cycle starting at~$U$ must be finite and end outside~$U$. 

Thus we can successively do every possible flip on the sink faces of~$\eta$ and~$\nu$, to obtain~$\eta'$ and~$\nu'$ which are going inside~$U$ along~$\partial U$. Then we can compare~$\eta'$ and~$\nu'$ on the connected components of~$S\setminus U$. We will adapt the previous procedure to find a sequence of flip from~$\eta$ to~$\nu$. Define~$E$ in the same way, then~$E$ delimitates connected regions. For every such region~$F$, either~$U$ is outside~$F$ and we apply the flip algorithm on the sub-faces of~$F$, or~$U$ is inside~$F$, and we apply the flip algorithm on~$S\setminus F$. Hence we can apply a sequence of elementary flips to eliminate~$\partial F$ from~$E$, and successively transform~$\eta'$ into~$\nu'$.

Finally suppose that~$\eta$ is not acyclic and that~$U$ is not connected. We will construct~$\nu$ cohomologous to~$\eta$, and a non-closed geodesic that intersects finitely~$\Sigma_\eta$ and~$\Sigma_\nu$ but not with the same amount. Denote~$U_1,\cdots, U_n$ the connected components of~$U$. Notice that~$\eta$ partially order~$(U_i)_i$. Indeed suppose there is a finite sequence of oriented paths connecting~$U_{i_1}$ to~$U_{i_1},\cdots U_{i_k}$ and back to~$U_{i_1}$, then it is included in an oriented cycle intersecting~$U_k$, that must remain included in~$U_k$. 

We successively do every possible flip on sink faces (which terminates in a finit number of steps), to obtain~$\eta'$. Let~$F$ be a connected component of~$\Gamma^\star\setminus U$. Then every increasing path in~$(F^\star,\eta')$ is finite, and end at the boundary of~$F$. Since the~$U_i$ are partially ordered, there is a maximal~$U_k$. And since there is no path inside~$F$ starting and ending at~$U_k$,~$\eta'$ is going inward~$U_k$ along its boundary. Let~$\nu$ be the coorientation obtained from~$\eta'$ by changing the co-orientation of~$U_k$. Then~$\nu$ is Eulerian and cohomologous to~$\eta'$ and~$\eta$. 

\begin{figure}[h]
	\definecolor{purple}{RGB}{255,0,255}
	\centering
	\begin{center}
	\begin{picture}(123,40)(0,0)
	\put(0,0){\includegraphics[height=40mm]{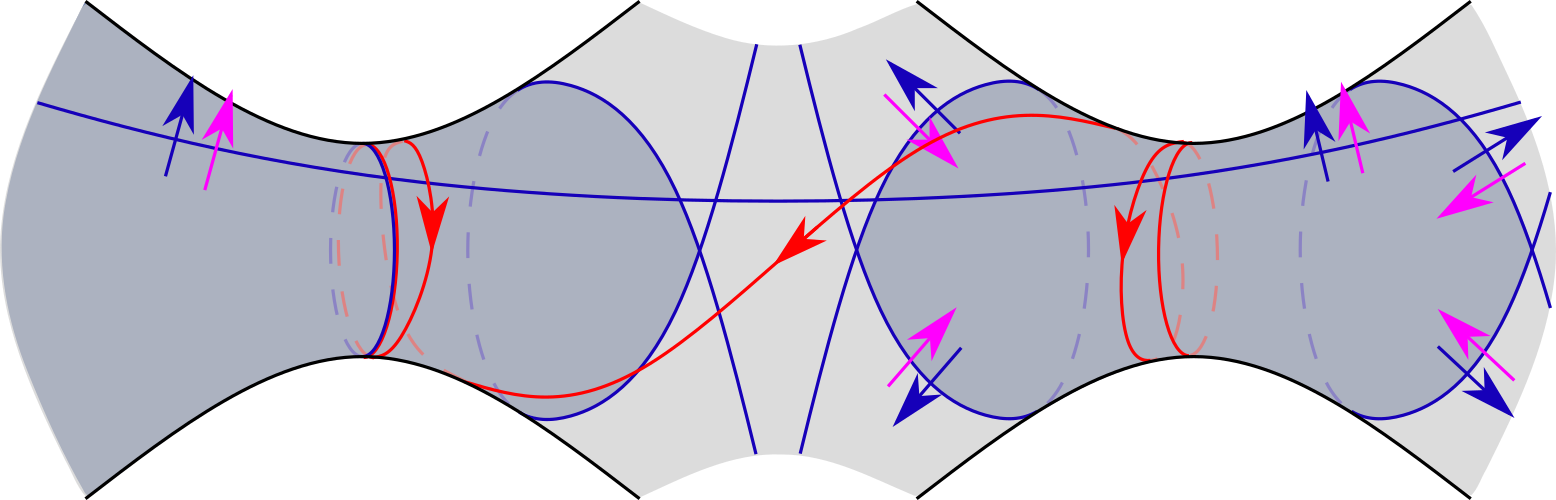}}
	
	\put(61,38){$S$}
	\color{blue}
	\put(61,30){$\Gamma$}
	\put(72,4){$\eta$}
	\color{blue!40!gray}
	\put(29,7){$U_i$}
	\put(95,7){$U_k$}
	\color{red}
	\put(63,16){$\delta$}
	\color{purple}
	\put(72,14){$\nu$}

	\end{picture}
	\end{center}
	\caption{Non closed geodesic intersecting~$\Sigma_\eta$ and~$\Sigma_\nu$ a different amount, but with~$[\eta]=[\nu]$.}
	\label{fig:NonIsotopy}
	\end{figure}
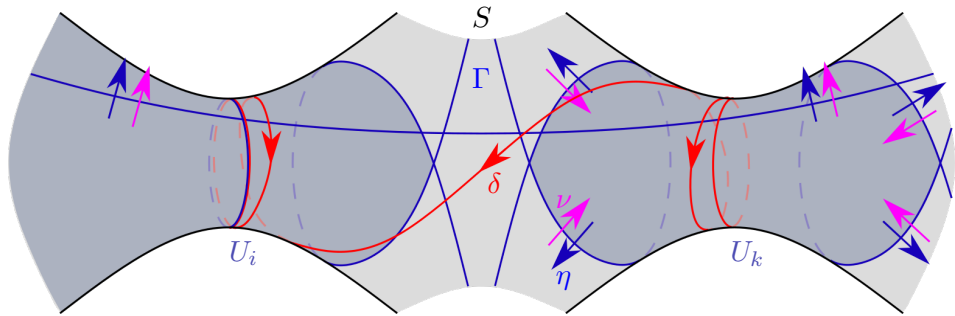

By Lemma \ref{acyclic}, for every~$1\leq i\leq n$, there exists a geodesic~$\delta_i$ inside~$U_i$ (or on its boundary). Let~$1\leq i\leq n$ be different from~$k$, and define a non-closed geodesic~$\delta$ as in Figure \ref{fig:NonIsotopy}, so that~$\delta$ accumulates in the infinite past along~$\bar{\delta_i}$ ($\delta_i$ with the opposite direction), and accumulates in the infinite future along~$\bar{\delta_k}$. Since~$i\neq k$, the algebraic intersection is~$\delta\cap\partial U_k$ is odd. We can do this so that~$\delta$ remains inside the interior of~$U_k\cup U_i$ outside a compact arc. But if~$\delta_i\not\subset\partial U_i$ then~$\larrow{\delta_i}\cap\Sigma\eta=\emptyset$, and if~$\delta_i\subset\partial U_i$ then~$\larrow{\beta_i}\cap\Sigma\eta=\emptyset$ where ~$\beta_i$ is any slight pushed of~$\delta_i$ inside~$U_i$. Thus~$\delta\cap \Sigma_{\eta'}$ and~$\delta\cap \Sigma_\nu$ are finite, and differ by an odd integer, that is not~$0$. Thus~$\Sigma_\eta$ and~$\Sigma_\nu$ are not isotopic through the flow.

\end{proof}

\end{appendices}

\end{document}